\documentclass[11pt]{article}
 \usepackage{amsmath, amssymb, amsfonts, amsthm, enumerate, tikz, a4wide}
\date{}
\usepackage[margin= 21mm,bottom=22mm, top= 20mm]{geometry}

\newcommand\blfootnote[1]{%
  \begingroup
  \renewcommand\thefootnote{}\footnote{#1}%
  \addtocounter{footnote}{-1}%
  \endgroup
}

\newcommand{\Comment}[1]{\textbf{[#1]}}
\renewcommand{\Comment}[1]{}

\newcommand{\Hide}[1]{}

\newtheorem{theorem}{\bf Theorem}[section]
\newtheorem*{theorem*}{\bf Theorem}

\newtheorem{conjecture}[theorem]{\bf Conjecture}

\newtheorem{definition}[theorem]{\bf Definition}

\newtheorem{fact}[theorem]{\bf Fact}

\newtheorem{claim}[theorem]{\bf Claim}
\newtheorem{lemma}[theorem]{\bf Lemma}

 \title{Notes on embedding   trees in graphs with $O(|T|)$-sized covers}
\date{}

\author{
Alexey Pokrovskiy}

\begin{document} 
\maketitle
\begin{abstract}
This is a companion paper to the paper ``Hyperstability in the Erd\H{o}s-S\'os Conjecture''. In that paper the following rough structure theorem was proved for graphs $G$ containing no copy of a bounded degree tree $T$: from any such $G$, one can delete $o(|G||T|)$ edges in order to get a subgraph each of whose connected components have a vertex cover of order $3|T|$.

This theorem creates an incentive for studying graphs whose connected components have covers of order $O(|T|)$ --- and this is what will be explored here. It turns out that such graphs are amenable to regularity approaches which have been successful in studying dense  $T$-free graphs. In this paper we will follow such an approach from the paper ``On the Erd\H{o}s–S\'os conjecture for trees with bounded degree'' by  Besomi, Pavez-Sign\'e, and Stein and show how it can be adapted from dense graphs to graphs with a small vertex cover.
\end{abstract}
\section{Introduction}
\blfootnote{Department of Mathematics, University College London, Gower Street, London WC1E 6BT, UK. \\ \emph{Email}: \texttt{dralexeypokrovskiy@gmail.com.}}
One of the biggest successes of modern combinatorics has been Szemeredi's Regularity Lemma --- a tool which gives a unified approach for studying all sorts of questions about dense graphs. Here ``dense graph'' means a graph $G$ where $e(G)=\Omega(|G|^2)$.  This lemma provides a way to approximate large dense graphs by small weighted graphs.
A notable example where Szemerédi’s Regularity Lemma plays a key role is in the context of the Erd\H{o}s-S\'os Conjecture. This conjecture predicts  what the Tur\'an number of trees should be: 
\begin{conjecture}[Erd\H{o}s, S\'os, see \cite{erdos1964extremal}]\label{conj:Erdos_Sos}
Every graph $G$ with $> (d-1)n/2$ edges contains a copy of every $d$-edge tree $T$.
\end{conjecture}
There are several different constructions which are extremal and near-extremal for this conjecture. One is to consider $n/d$ vertex-disjoint complete subgraphs of order $d$ --- this has $(n/d)\binom {d}2=(d-1)n/2$ but no copies of $T$ (since the connected components don't have enough vertices to house $T$). If $T$ is a star, then any $(d-1)$-regular graph contains no $T$ and has $(d-1)n/2$ edges.  If $T$ is balanced (i.e. both parts of its bipartition have order $(d+1)/2$) then consider $G$ be a graph containing $(d-1)/2$ vertices which are joined to everything, and no other edges. Then $e(G)=n(d-1)/2-(d-1)(d-5)/8$, which, for large $n$, shows that $G$ is near-extremal for the Erd\H{o}s-S\'os Conjecture.

Much is known about the conjecture. In 1959, predating the conjecture, Erd\H{o}s and Gallai proved it for paths. For stars, the conjecture is an easy exercise. There are more early results about specific classes of trees~\cite{borowiecki1993erdos}, and about $n=d+1$~\cite{zhou1984note}. 
In the early 1990s, Ajtai, Komlós, Simonovits, and Szemerédi announced a complete proof of the Erdős-Sós Conjecture for sufficiently large $d$. The statement of their result is as follows:
\begin{theorem}[Ajtai, Koml\'os, Simonovits, and Szemer\'edi]  For sufficiently large $d$, every graph $G$ with $> (d-1)n/2$ edges contains a copy of every $d$-edge tree $T$.
\end{theorem}

Although the proof of this theorem has not yet appeared, several details about its approach are known. The proof strategy involves an adapted version of Szemer\'edi's Regularity Lemma, tailored to sparse graphs, and a careful embedding of parts of the tree into regular pairs of the graph. This method is supplemented by stability and extremal case analysis to achieve an exact result. Over the intervening decades, interest in replicating or extending this proof has grown, particularly given that the Erd\H{o}s-S\'os Conjecture is not an isolated problem but has numerous variants (see surveys like~\cite{stein2020tree}).

For dense graphs, significant progress has been made on understanding the conjecture and the techniques of Ajtai, Koml\'os, Simonovits, and Szemer\'edi. Rozho\v{n}~\cite{rozhon2019local} and Besomi, Pavez-Sign\'e, and Stein~\cite{besomi2021erdHos} independently proved approximate versions of the conjecture for trees with bounded maximum degree. Moreover, Besomi, Pavez-Sign\'e, and Stein\cite{besomi2021erdHos} extended their work to an exact result for dense graphs when the tree has bounded degree. Very recently, a proof of the approximate version of the conjecture has been announced for dense graphs without any degree restrictions by Davoodi, Piguet, {\v{R}}ada,  and Sanhueza-Matamala~\cite{davoodi2023beyond}. See the survey~\cite{stein2020tree} for a more detailed history of the conjecture.

From the above we see that we have a pretty good understanding of the conjecture of dense graphs. But sparse graphs, our understanding of how to apply the techniques of Ajtai, Koml\'os, Simonovits, and Szemer\'edi remains incomplete. 
The author recently proved the following theorem which is tailored specifically to understanding the sparse case of the Erd\H{o}s-S\'os Conjecture and its relatives. Recall that a cover in a graph is a set of vertices $C$ which has the property that all edges have one or both of their vertices in $C$.
\begin{theorem}[\cite{hyperstability}]\label{main_theorem} Let $\Delta \in \mathbb{N}$ and $\varepsilon > 0$, and let $d$ be sufficiently large. Let $T$ be a tree with $d$ edges and $\Delta(T) \leq \Delta$. For any graph $G$ with $e(G) \geq \varepsilon d |G|$ having no copies of $T$, it is possible to delete $\varepsilon e(G)$ edges to obtain a graph $H$, each of whose connected components has a cover of order $\leq 3d$. \end{theorem}

This theorem allows us to analyse the Erd\H{o}s-S\'os Conjecture even in the sparse case by reducing the problem to studying graphs with small covers.  While graphs with covers of order $\le 3d$ are not necessarily dense, they still turn out to be amenable to various techniques used for studying dense graphs. See Sections~1.1 -- 1.3 of~\cite{hyperstability} for some simple examples of this. The current paper is concerned with giving more complicated ways of studying graphs  with small covers --- in particular studying them using Szemer\'edi's Regularity Lemma. 
 Our results and their proofs are all heavily inspired by the paper of Besomi, Pavez-Sign\'e, and Stein~\cite{besomi2021erdHos} where the same sorts of results were proved for dense graphs. Our first result is a stability theorem --- a theorem which says that if $G$ has slightly fewer edges than the Erd\H{o}s-S\'os Conjecture asks for, then $G$ is close in structure to one of the near-extremal constructions in that conjecture. 
\begin{theorem}[Stability theorem]\label{thm:stability_main}
Let $\Delta\gg  \alpha \gg \varepsilon\gg  k^{-1}$.
Let $T$ be a $k$-vertex tree with $\Delta(T)\leq \Delta$.
Let $G$ be a graph with $e(G)\geq (1-\varepsilon)kn/2$, and let $D$ be a cover   of order $|D|\leq 10k$. If $G$ has no copy of $T$, then there is a subgraph $H$ of $G$ 
satisfying one of:
\begin{enumerate}[(i)]
\item $\delta(H)\geq (1-\alpha)k$ and $|H|\leq (1+\alpha)k$.
\item  $H$ is bipartite with parts $X,Y$ such that $|X|\leq (1+\alpha)k/2$, $|Y|\geq 6k$, $\delta(X)\geq (1-\alpha)|Y|$, $\delta(Y)\geq (1-\alpha)k/2$.
\end{enumerate}
\end{theorem}
Stability theorems are often paired with extremal case analysis results that prove Conjecture~\ref{conj:Erdos_Sos} when $G$ contains contains the structure provided by the stability theorem. The first of these we prove deals with when $G$ has an almost-complete subgraph as provided by (i) above. 

\begin{theorem}[Almost-complete extremal analysis]\label{thm:nonbipartite_extremal}
Let $\Delta^{-7}/1000\geq \varepsilon\ge 3k^{-1}$.
Let $G$ be a connected graph with $|G|\geq (1+256\Delta^2\sqrt{\varepsilon})k$, $\delta(G)\geq 128\Delta^2\sqrt{\varepsilon} k$ containing a subgraph $K$ with $|K|\leq (1+\varepsilon)k$, $\delta(K)\geq (1-\varepsilon)k$. Then $G$ contains a copy of every $k$-edge tree having $\Delta(T)\leq \Delta$.
\end{theorem}
The second extremal case analysis result that we prove deals with the case when $G$ has an almost-complete bipartite subgraph as provided by (ii) of Theorem~\ref{thm:stability_main}.

\begin{theorem}[Almost-complete bipartite extremal analysis]\label{thm:bipartite_extremal}
Let $\Delta^{-7}/1000\geq \varepsilon\ge k^{-1}$.
Let $T$ be a $k$-edge tree with $\Delta(T)\leq \Delta$.
Let $G$ be a connected graph with $\delta(G)\geq k/2$. Suppose that $G$ contains a bipartite subgraph $B$ with parts $X, Y$ such that $|X|\leq (1+\varepsilon)k/2$, $|Y|\geq 6k$, $\delta_B(X)\geq (1-\varepsilon)|Y|$, $\delta_B(Y)\geq (1-\varepsilon)k/2$. 
Then $G$ contains a copy of $T$.
\end{theorem}
The combination of these theorems together with Theorem~\ref{main_theorem} allows one to  prove the Erd\H{o}s-S\'os Conjecture for large bounded degree trees. This proof can be found in~\cite{hyperstability}.

Additionally, we prove a stronger version of Theorem~\ref{main_theorem}  which sharpens the constant ``$3$'' to the asymptotically optimal ``$1+\varepsilon$''.
\begin{theorem}\label{optimal_theorem} Let $\Delta, \varepsilon\gg d^{-1}$. Let $T$ be a tree with $k$ edges and $\Delta(T) \leq \Delta$. For any graph $G$ with $e(G) \geq \varepsilon k |G|$ having no copies of $T$, it is possible to delete $\varepsilon e(G)$ edges to obtain a graph $H$, each of whose connected components has a cover of order $\leq (1 + \varepsilon)k$. \end{theorem}
This bound can seen to be near-optimal by considering a graph $G$ of vertex disjoint cliques of order $k-1$. This $G$  has no copy of $T$ because its components are too small. But by deleting $\varepsilon e(G)$ edges, at least one of the cliques will have $\geq (1-\varepsilon)\binom{k-1}2$ edges left --- and so has no cover smaller than $(1-2\sqrt{\varepsilon} )k$ (the maximum number of edges that can touch $(1-2\varepsilon )k$ vertices in a subgraph of $K_{k-1}$ is $\binom{k-1}2-\binom{2\sqrt{\varepsilon} k}2< (1-\varepsilon )\binom{k-1}2$).

The proof of this theorem doesn't involve modifying the original proof of Theorem~\ref{main_theorem}. Instead, Theorem~\ref{main_theorem} is used as a black box, and on top of it one adds some regularity-type arguments in the spirit of those used by  Besomi, Pavez-Sign\'e, and Stein.

\section{Preliminary lemmas}
Here we introduce some preliminary lemmas. For convenience, we refer to~\cite{hyperstability} for the proofs of all lemmas that we state without proof (however it is worth emphasising that~\cite{hyperstability} is not the source of any of these lemmas and variants of them are much older).

For a graph $G$ and a set of vertices $S$, we use $G\setminus S$ to   denote $G$ with the vertices of $S$ deleted. For a set of edges $T$,  we use $G\setminus T$ to   denote $G$ with the edges of $T$ deleted.
In particular two graphs $G,H$, we use $G\setminus V(H)$ to denote $G$ with the vertices of $H$ deleted, $G\setminus E(H)$ to denote $G$ with the edges of $H$ deleted, $G\cap H$ to denote the graph with  edge set $E(G)\cap E(H)$ and isolated vertices deleted.
We use $G\setminus H$ to denote the graph with edge set $E(G)\setminus E(H)$ and isolated vertices deleted.
We use $G\cup H$ to denote the graph on $V(G)\cup V(H)$ with edge set $E(G)\cup E(H)$.
For sets, $S,T$ in a graph $G$, we define $\delta_S(T):=\min_{v\in T} |N_G(v)\cap S|$.
Given a family $\mathcal F=\{G_1, \dots, G_t\}$ of graphs, we define the graph $\bigcup \mathcal F:=G_1\cup \dots \cup G_t$. 
An embedding of a graph $G$ into a graph $H$ is an injective homomorphism $f:G\to H$.

For two (possibly overlapping) sets of vertices $S,T$ in a graph $G$, let $e_G(S,T)$ denote the number of edges  in $G$ starting in $S$ and ending in $T$ (counting edges contained in $S\cap T$ only once). Note that for any $S,T$, we have $2e(S,T)\geq \sum_{s\in S} |N(s)\cap (T\cup S)|$, with equality when $S=T$. We will refer to this as the ``handshaking lemma'', since it is a slight variant of the usual lemma.

For non-negative $x,y$, we write $a=x\pm y$ to mean $x-y\leq a\leq x+y$.   
 We use standard notation for hierarchies of constants, writing $x\gg y$ to mean that there is a   function $f : (0,1] \rightarrow (0, 1]$ such that all subsequent statements hold for $x\geq f(y)$.

We'll need the standard fact that every graph $G$ has a subgraph whose minimum degree is at least half the average degree of $G$.
\begin{fact}[\cite{hyperstability}]\label{fact:subgraph_min_degree_e/2}
Every graph $G$ with $e(G)\geq x|G|$ has a subgraph $H$ with $e(H)\geq x|H|$ and minimum degree $\delta(H)\geq \lfloor x\rfloor+1\geq x$.
\end{fact}

We'll need a standard fact that trees can be greedily embedded in graphs of high minimum degree.
\begin{fact}[\cite{hyperstability}]\label{fact:greedy_tree_embedding}
Let $G$ be a graph, $v\in V(G)$ with $\delta(G-v)\geq d$ and $d(v)\geq \Delta$. Let $T$ be  $\leq d$-edge  tree with $\Delta(T)\leq \Delta$. Then $G$ has a copy of $T$ rooted at $v$.
\end{fact}

We'll need that in any graph there's a connected component which has at least the same average degree as the whole graph. 
\begin{fact}\label{fact:connectected_component_with_same_density_as_G}
Every graph $G$ has a connected component $C$ with $e(C)/|C|\geq e(G)/|G|$.
\end{fact}

We'll need the fact that in a bounded degree tree both parts of the bipartition are reasonably large. 
\begin{fact}\label{fact:large_bipart_in_bounded_degree_tree}
In any tree on $\geq 2$ vertices, the parts of the bipartition have order $\geq |V(T)|/2\Delta(T)$.
\end{fact}
\begin{proof}
Letting $X$ be a part of the bipartition, we have $|V(T)|-1=e(T)=\sum_{x\in X}d(x)\leq \Delta(T)|X|$. Rearranging gives $|X|\geq \frac{|V(T)|-1}{\Delta(T)}\geq \frac{|V(T)|}{2\Delta(T)}$.
\end{proof}

We'll need the following simple fact about integer parts.
\begin{fact}\label{fact:ceilfloor}
For any $k\in \mathbb{\mathbb{Z}}$, $\lfloor\frac{k+1}2\rfloor=\lceil\frac{k}{2} \rceil$.
\end{fact}
\begin{proof}
If $k$ is even, $\lceil\frac{k}{2} \rceil=\frac k2\le\frac{k+1}2$ and $\lceil\frac{k}{2} \rceil+1=\frac k2+1=\frac{k+2}{2}>\frac{k+1}2$. Thus $\lceil\frac{k}{2} \rceil$ is the largest integer $\le \frac{k+1}2$ i.e. $\lfloor \frac{k+1}2\rfloor=\lceil\frac{k}{2} \rceil$.

If $k$ is odd, then $\lfloor\frac{k+1}2\rfloor=\frac{k+1}2\ge \frac k2$ and $\lfloor\frac{k+1}2\rfloor-1=\frac{k+1}2-1=\frac{k-1}{2}< \frac k2$.  Thus $\lfloor\frac{k+1}{2} \rfloor$ is the smallest integer $\ge \frac{k}2$ i.e. $\lceil\frac{k}{2} \rceil=\lfloor \frac{k+1}2\rfloor$.
\end{proof}

\subsection{Cut-density}
A lot of our intermediate lemmas will be stated via the notion of cut-density, a concept which was first introduced by Conlon, Fox, and Sudakov~\cite{conlon2014cycle}.
\begin{definition}
A graph $G$ is a $q$-cut-dense if every partition $V(G)=A\cup B$ has $e_G(A,B)\geq q|A||B|$.
\end{definition}

Roughly speaking, cut-density  combines two other properties of graphs --- namely connectivity and having high minimum degree. The latter comes from the following: 
\begin{fact}[\cite{hyperstability}] \label{fact:cut_dense_min_degree} 
Let $G$ be a $p$-cut-dense order $n\geq 10$ graph. Then $\delta(G)\geq p(n-1)\geq 0.9pn$ and $e(G)\geq pn^2/4$
\end{fact}

The following two lemmas explain how to get larger cut-dense subgraphs from smaller ones.
\begin{lemma}[\cite{hyperstability}] \label{lem:cut_dense_union} 
 Let $G_1,G_2$ be  $q$-cut-dense. Then $G_1\cup G_2$ is $\frac{q|V(G_1)\cap V(G_2)|}{4|V(G_1)\cup V(G_2)|}$-cut-dense.
\end{lemma}

\begin{lemma}[\cite{hyperstability}] \label{lem:cut_dense_add_vertices} 
Let $\delta \leq 1$.
Let $G$ be  $q$-cut-dense and $H$ a graph containing $G$ such that  $|N_H(v)\cap V(G)|\geq \delta|G|$ for all $v\in V(H)\setminus V(G)$. 
 Then $H$ is   $(\frac{q\delta|G|^2}{4|H|^2})$-cut-dense.
\end{lemma}

All dense graphs can be approximately decomposed into cut-dense components.
\begin{lemma}[\cite{hyperstability}]\label{lem:cut_dense_decomposition} 
Let $G$ be a graph. It is possible to delete $q n^2$ edges from $G$ to get a subgraph all of whose components are $q$-cut-dense.
\end{lemma}

A consequence of this is that all dense graphs  contain a large cut-dense subgraph.
\begin{lemma}\label{lem:find_one_cutdense_nonexact} 
Let $G$ be a graph with $e(G)\geq 2qn^2$. Then $G$ has a $q$-cut-dense subgraph of order $\geq qn$.
\end{lemma}
\begin{proof}
Apply Lemma~\ref{lem:cut_dense_decomposition} to get a subgraph $G'$ with $e(G')\geq q n^2$, all of whose components are $q$-cut-dense. By the handshaking lemma, $2qn^2=\sum_{v\in V(G)} d_{G'}(v)$, and so, by the Pigeonhole Principle, some vertex $v$ has degree $\geq 2qn$. This vertex is in a $q$-cut-dense component of order $\geq qn$ as required.
\end{proof}

One weakness of cut-density is that it only makes sense for dense graphs. When working with Theorem~\ref{main_theorem}, one would like to have an analogous notion of connectivity of sparse graphs (that have a small cover).  One way of creating such a notion is to ask that the small cover is contained in a cut-dense graph (which forces various connectivity notions on the whole graph). The following two lemmas show that graphs that have small covers can approximately be decomposed into vertex-disjoint subgraphs like this.
\begin{lemma}\label{lem:cut_dense_decomposition_of_dominated_graphs}
Let $p\le 1/1000$ and $k\ge 10^6$. 
Let $G$ be bipartite with parts $X$, $Y$ with $|X|\leq 10k$.  There are  vertex-disjoint subgraphs $G_1, \dots, G_t$ covering all but $\leq 200p kn$ edges of $G$ such that each $G_i$ has a subgraph $D_i$ with $p^4k\le |D_i|\leq   420 p^{-8}k$, $D_i$ $p^{10^8p^{-7}}$-cut-dense, and $|N(v)\cap V(D_i)|\geq p^{10^6p^{-1}}k$ for all $v\in G_i$.
\end{lemma}
\begin{proof}
Let $\varepsilon=p^{10^6p^{-1}}$ and let $n:= |V(G)|$. 
Let $H_1, \dots, H_t$ be a maximal collection of vertex-disjoint $p^2$-cut-dense subgraphs in $G$ of order $\ge  p^2k$. Since each $H_i$ has $\ge\delta(H_i)\overset{\text{F.~\ref{fact:cut_dense_min_degree}}}{\geq}   0.5p^2|H_i|\ge 0.5p^2\cdot p^2k\ge 0.5p^6k$ vertices in $X$, we have that $t\leq \frac{|X|}{0.5p^6k}\le  \frac{10k}{0.5p^6k}= 20p^{-6}$. Since vertices in $Y$ have degree $\le |X|\le 10k$, we have $10k\geq \delta(H_i)\geq 0.5p^2|H_i|$, giving $|H_i|\le 20p^{-2}k$. 
For a graph $H$ define $\kappa(H)$ to be the maximum $\kappa$ for which $H$ is $\kappa$-cut-dense. Repeat the following: as long as there are two graphs $H^-, H^+$ in the collection contained in some graph $H\subseteq  V(G)\setminus\bigcup_{H'\neq H^-, H^+} V(H')$ (i.e. vertex-disjoint from the rest of the collection) with $\kappa(H)\geq \varepsilon^2p^{22}\min(\kappa(H^-),\kappa(H^+))$ and $|H|\leq |H^-\cup H^+|+k$, then replace $H^-, H^+$ in the collection with $H$. 
Let $D_1, \dots, D_s$ be the final collection of graphs we get. Note that since the joining operation happens $\leq t$ times, we have $\kappa(D_i)\geq (\varepsilon^2p^{22})^t p^2\ge(\varepsilon^2p^{22})^{20p^{-6}} p^2  = p^{(10^6p^{-1}+22)20p^{-6}+2}>p^{10^8p^{-7}}$ and $|D_i|\leq t( 20p^{-2}k+k)\le 20p^{-6}\cdot 21p^{-2}k= 420 p^{-8}k$ for all $i$. Also note that $D_1, \dots, D_s$ are vertex-disjoint (the initial $H_1, \dots, H_k$ were vertex-disjoint and at every step of the joining operation, disjointness was preserved).
For each $i$, let $V_i=\{y\in V(G)\setminus  \bigcup V(D_i): |N(y)\cap V(D_i)|\geq \varepsilon k\}$.  Let $V_0=V(G)\setminus \bigcup V_i\cup D_i$. 

\begin{claim}
 $|V_i\cap (V_j\cup D_j)|\leq \varepsilon k$ for all $i\neq j$. 
\end{claim}
\begin{proof} 
 Suppose otherwise, and let $U\subseteq V_i\cap (V_j\cup D_j)$ be a subset of order $\varepsilon k$. By Lemma~\ref{lem:cut_dense_add_vertices} $D_i\cup U$ and $D_j\cup U$ are $\ge \frac{\min(\kappa(D_i), \kappa(D_j))(\varepsilon k/420 p^{-8}k) (p^2k)^2}{4(p^2k+\varepsilon k)^2}\geq \varepsilon p^{13}\min(\kappa(D_i), \kappa(D_j))$-cut-dense. By Lemma~\ref{lem:cut_dense_union}, $D_i\cup D_j\cup U$ is $\frac{\varepsilon p^{13}\min(\kappa(D_i), \kappa(D_j))(\varepsilon k)}{4(2(420 p^{-8}k)+\varepsilon k)}\geq \varepsilon^2 p^{22}\min(\kappa(D_i), \kappa(D_j))$-cut-dense. Since $D_i\cup D_j\cup U\subseteq  V(G)\setminus \bigcup_{r\neq i,j}D_r$, this contradicts the fact that we can't join $D_i, D_j$. 
\end{proof} 
 The above claim implies that there are $\leq \varepsilon kn$ edges touching $V_i\cap (V_j\cup D_j)$.

\begin{claim}
There are $\leq 100p k n$ edges touching $V_0$. 
\end{claim}
\begin{proof}
Suppose otherwise, and let $G'=G[V_0, V(G)\setminus \bigcup D_i]$. Using the definition of $V_i$, for every  $v\in V_0\subseteq V(G)\setminus (V_i\cup \bigcup D_j)$ there are $<\varepsilon k$ edges from $v$ to $D_i$. This shows that there are $\leq \varepsilon t k |V_0|<\varepsilon(20p^{-6})kn<pkn$ edges from $V_0$ to $\bigcup D_i$, giving $e(G')\geq 99pkn$.  Note that $G'$ has a vertex in $X$ of degree $\geq e(G')/|X|\geq 99pkn/n= 99pk$, and so $|V(G')\cap Y|\ge 99pk$. Select a set $\hat Y$ of $99pk\le k$ vertices of $G'\cap Y$ uniformly at random to get a graph $G''=G'[\hat Y\cup (G'\cap X)]$ with $\le 11k$ vertices and expected number of edges $\geq 99 pk n(99pk/n)\geq  3p^2(11k)^2\geq 3p^2|G''|^2$. Fix one such graph.  Lemma~\ref{lem:find_one_cutdense_nonexact} gives a new $p^2$-cut-dense subgraph vertex-disjoint from the original $H_1, \dots, H_t$ with $|H_i|\ge p^2k$ (contradicting maximality).
\end{proof}

For $i=1, \dots, t$, let $G_i=D_i\cup V_i\setminus (\bigcup_{j\neq i} D_j\cup V_j)$, noting that this ensures that $G_1, \dots, G_t$ are vertex-disjoint. Edges outside $\bigcup G_i$ must touch $V_0\cup \bigcup V_i\cap (V_j\cup D_j)$, --- and we've established that there are $\leq 100pkn+ \varepsilon knt\leq 100pkn+\varepsilon(20p^{-6})kn\leq 200pkn$ such edges.
\end{proof}

The following lemma is a slight strengthening of the previous one. that makes the $D_i$s contain all the $X$-vertices of the corresponding $G_i$.
\begin{lemma}\label{lem:cut_dense_decomposition_of_dominated_graphs2}
Let $p\le 1/1000$ and $k\ge 10^{6}$. 
Let $G$ be bipartite with parts $X$, $Y$ with $|X|\leq 10k$.  There are  vertex-disjoint subgraphs $G_1, \dots, G_t$ covering all but $\leq 200p kn$ edges of $G$ such that each $G_i$ has a subgraph $D_i$ with $p^4k\le |D_i|\leq   430 p^{-8}k$, $D_i$ $p^{10^9p^{-8}}$-cut-dense, $|N(v)\cap V(D_i)|\geq p^{10^6p^{-1}}k$ for all $v\in G_i$, and $V(G_i)\cap X\subseteq V(D_i)$.
\end{lemma}
\begin{proof}
Apply Lemma~\ref{lem:cut_dense_decomposition_of_dominated_graphs} to get families $G_1, \dots, G_t$ and $D_1, \dots, D_t$. For each $i$, let $X_i=G_i\cap X\setminus D_i$ and let $D_i':=G_i[V(D_i)\cup X_i]$. Note  $p^4k\le |D_i|\le |D_i'|\le |D_i|+10k\le 420 p^{-8}k+10k\le430 p^{-8}k$. By Lemma~\ref{lem:cut_dense_add_vertices}, $D_i'$ is  $\frac{(p^{10^8p^{-7}})(p^{10^6p^{-1}}k/420 p^{-8}k)(p^4 k)^2}{4(430 p^{-8}k)^2}\geq  p^{10^9p^{-8}}$-cut-dense. We have $|N(v)\cap V(D_i')|\ge |N(v)\cap V(D_i)|\geq p^{10^6p^{-1}}k$ for all $v\in G_i$,  and $V(G_i)\cap X=X_i\cup (D_i\cap X)\subseteq D_i'$ as required.
\end{proof}

\section{Stability analysis}
In this section, we prove Theorem~\ref{thm:stability_main}. We start by introducing the fundamentals of Szemer\'edi's Regularity Lemma  in the next section. Afterwards we go over the techniques of Besomi, Pavez-Sign\'e, and Stein and extract some useful lemmas from there. Then we use these to prove some stability lemmas, culminating in Theorem~\ref{thm:stability_main}.
\subsection{Regularity preliminaries}
We begin by introducing the basics of the regularity method. 
A pair of disjoint vertex sets $(A,B)$ in a graph $G$ is said to be $\varepsilon$-regular if for all subsets $A'\subseteq A, B'\subseteq B$ having $|A'|\ge \varepsilon|A|, |B'|\geq \varepsilon|B|$, we have $e_G(A', B')=(1\pm \varepsilon)|A'||B'|\frac{e_G(A,B)}{|A||B|}$. 
An $(\varepsilon, \eta)$-regularity partition of $G$ is a partition $V(G)=V_1, \dots, V_l$ with:
\begin{itemize}
\item $|V_1|=\dots=|V_l|$.
\item $V_i$ is independent in $G$.
\item For $i\neq j$, the pair $(V_i,V_j)$ is $\varepsilon$-regular with density either $\geq \eta$ or $=0$.
\end{itemize}
The $(\varepsilon, \eta)$-reduced graph of the partition is the graph $R$ on $V(R)=\{V_1, \dots, V_l\}$ with $V_iV_j$ an edge whenever $d(V_i, V_j)\geq \eta$. We now state the degree form of the regularity lemma (see~\cite{komlos1995szemeredi}, Theorem 1.10).  

\begin{theorem}[Szemer\'edi's Regularity Lemma]\label{thm:regularity_lemma}
Let $\varepsilon, m^{-1}\gg M^{-1}$. Every graph $G$ has a subgraph $G'$ such that $|G|-|G'|\leq \varepsilon n$, every vertex $v\in G'$ has $d_G(v)-d_{G'}(v)\leq (\varepsilon+\eta)n$, and $G'$ has an $(\varepsilon, \eta)$-regular partition where the number of parts is in $[m,M]$.
\end{theorem}
We say that a partition $V_1, \dots , V_k$ \emph{refines} another partition $U_1, \dots, U_m$ if for all $i$, we have $V_i\subseteq U_j$ for some $j$. It is a standard fact that in Theorem~\ref{thm:regularity_lemma}, one specify some initial partition $V(G)=U_1\cup \dots \cup U_m$, and then ensure that the partition produced by the lemma refines this. 

Note that in this theorem $e(G)-e(G')\leq \sum_{v\in V(G)}(d_G(v)-d_{G'}(v))\leq (|G|-|G'|)n+\sum_{v\in V(G')}(d_G(v)-d_{G'}(v))\leq \varepsilon n^2+\sum_{v\in V(G')}(\varepsilon+\eta)n\leq (2\varepsilon+\eta)n^2$.

\begin{fact}\label{fact:cut_dense_preserved_by_regularity} 
Let $G$ be an $n$-vertex, $q$-cut-dense graph and  $G'$ a subgraph such that every vertex $v\in G'$ has $d_G(v)-d_{G'}(v)\leq \alpha n$. Then $G'$ is $(q-2\alpha)$-cut-dense.
\end{fact}
\begin{proof}
Consider a partition $V(G')=U\cup V$. Without loss of generality $|U|\leq |V|$ giving $|V|\geq n/2$.  Let $V'=V\cup (V(G)\setminus V(G'))$ --- so now $U, V'$ is a partition of $V(G)$. We have $e_{G'}(U,V)=e_{G'}(U, V')=\sum_{u\in U}|N_{G'}(u)\cap V'|=\sum_{u\in U}(|N_{G}(u)\cap V'|-|(N_G(u)\setminus N_{G'}(u))\cap V'|)\geq \sum_{u\in U}|N_{G}(u)\cap V'|-|N_G(u)\setminus N_{G'}(u)| =  e_G(U, V')-\sum_{u\in U}(d_G(u)-d_{G'}(u))\geq q|U||V'|-\sum_{u\in U}\alpha n= q|U||V|-\alpha |U|n\geq q|U||V|-2\alpha |U||V|=(q-2\alpha)|U||V|$.
\end{proof}
This fact is useful in combination with the following.

\begin{lemma}\label{lem:cut_dense_connected_cluster_graph}
Let $q>\eta$.
Let $G$ be $q$-cut-dense, and $R$ a $(\eta, \varepsilon)$-reduced graph of $G$. Then $R$ is connected.
\end{lemma}
\begin{proof}
Let $V_1, \dots, V_t$ be the regularity partition corresponding to $R$. 
If $R$ is disconnected then we can partition $\{V_1, \dots, V_t\}=X\cup Y$ into non-empty sets so that there are no edges of $R$ from $X$ to $Y$. Let $X'=\bigcup_{V_i\in X} V_i$, $Y'=\bigcup_{V_j\in Y} V_j$ be the corresponding subsets of $G$. For $V_i\in X, V_j\in Y$, since $V_iV_j\not\in E(R)$, we have $e_G(V_i, V_j)<\eta|V_i||V_j|$. This gives $e_G(X', Y')=\sum_{V_i\in X}\sum_{V_j\in Y} e_G(V_i,V_j)< \sum_{V_i\in X}\sum_{V_j\in Y}\eta |V_i||V_j|=\eta|X'||Y'|$ (using that $\sum_{V_j\in Y}|V_j|=|Y'|$, $\sum_{V_i\in X}|V_j|=|X'|$ for the last equation). But by $q$-cut-density, we have $e_G(X', Y')\geq q|X'||Y'|> \eta |X'||Y'|$, a contradiction.
\end{proof}

We'll use a standard regularity inheritance lemma.
\begin{fact}\label{fact:regularity_inheritance} 
Let $\varepsilon<1/8$.
Let $(A,B)$ be an $(\varepsilon, \eta)$-regular pair, and $A'\subseteq A, B'\subseteq B$ be subsets of size $|A'|=|A|/2, |B'|=|B|/2$. Then $(A', B')$ is $(3\varepsilon, \eta/2)$-regular.
\end{fact}
\begin{proof}
By $\varepsilon$-regularity, we have $e(A',B')/|A'||B'|=(1\pm \varepsilon)e(A,B)/|A||B|$ which gives $e(A,B)/|A||B|=(1\pm 3\varepsilon/2)e(A',B')/|A'||B'|$ and $e(A',B')/|A'||B'|\geq (1-\varepsilon)\eta\geq \eta/2$. Let $A''\subseteq A', B''\subseteq B'$ have $|A''|\geq 2\varepsilon |A'|\geq \varepsilon |A|, |B''|\geq 2\varepsilon |B'|\geq \varepsilon|B|$. By $(\varepsilon, \eta)$-regularity, we have $e(A'', B'')/|A''||B''|=(1\pm \varepsilon)e(A,B)/|A||B|=(1\pm \varepsilon)(1\pm 3\varepsilon/2)e(A',B')/|A'||B'|=(1\pm 3\varepsilon)e(A',B')/|A'||B'|$.
\end{proof}

 A fractional matching in a graph is a function $f:E(G)\to[0,1]$ for which $\sum_{y\in N(v)}f(vy)\leq 1$ for all vertices $v$.  A fractional cover in a graph is a function $g:V(G)\to[0,1]$ for which $g(u)+g(v)\geq 1$ for all edges $uv$.
 The fractional matching number $\nu_f(G):=\max_{f \text{ fractional matching of G}} f(E(G))$ and the fractional cover number $\nu_f(G):=\max_{\text{$g$ fractional cover of G}} g(V(G))$ are known to be equal via linear programming duality. Moreover, it is known that optimal fractional matching/covers always exist whose values are always in the set $\{0, 1/2, 1\}$:
 \begin{theorem}[see \cite{scheinerman2013fractional}, Theorems~2.1.5 and~2.1.6]\label{thm:fractional_matching}
Let $G$ be a graph. There exists a fractional matching $f:E(G)\to \{0, 1/2, 1\}$ and a fractional cover $g:V(G)\to \{0, 1/2, 1\}$ so that $f(E(G))=\nu_f(G)=\tau_f(G)=g(V(G))$.
\end{theorem}
When working with regularity, it is well known that fractional matchings are as good as normal matchings.
\begin{lemma}\label{lem:matching_from_fractional_matching} 
Let $\varepsilon<1/2$.
Let $G$ be a graph and $R$ a $(\varepsilon, \eta)$-reduced graph of $G$ with $|R|\leq M_0$ in which the size of each part is even. Let $M$ be a fractional matching in $R$. Then there is a $(3\varepsilon, \eta/2)$-reduced graph $R'$ of $G$ with $|R|'=2|R|\leq 2M_0$, which has a matching $M'$ with $e(M')/|R'|=e(M')/2|R|=\nu_f(R)/|R|$.
\end{lemma}
\begin{proof}
Let $V_1, \dots, V_l$ be the partition giving rise to $R$. For each $i$, fix some partition $V_i=V_i^-\cup V_i^+$ into sets of equal size. Using Fact~\ref{fact:regularity_inheritance}, the partition of $G$ into   $V_1^-, V_1^+, \dots, V_l^-, V_l^+$ is $(3\varepsilon, \eta/2)$-regular, and hence we can define a corresponding $(3\varepsilon, \eta/2)$-reduced graph $R'$, noting that $|R'|=2l=2|R|$. Use Theorem~\ref{thm:fractional_matching} to get a fractional matching $f:E(R)\to \{0,1/2,1\}$    with $f(E(R))=\nu_f(R)$. Let $A=f^{-1}(1)$ and $B=f^{-1}(1/2)$, noting that $e(A)+e(B)/2=f(E(R))=\nu_f(R)$. Note that the condition ``$\sum_{y\in N(v)}f(vy)\leq 1$ for all vertices $v$'' ensures that in the graph $A\cup B$, vertices in $V(A)$ touch precisely one edge of $A$, while vertices in $V(B)$ touch at most two edges of $B$. This shows that $A\cup B$ is a vertex-disjoint union of a matching $A$ and a maximum degree $2$ subgraph $B$. Hence $B$ must be a union of paths and cycles. Construct a directed graph $D$ by having two copies of each edge in $A$ (one in each direction), and orienting the path/cycles in $B$ arbitrarily. Note that all vertices in $D$ have indegree/outdegree $\leq 1$ and $e(D)=2e(A)+e(B)$.
Let $M=\{V_i^-V_j^+: V_iV_j\in E(D)\}$, noting that this is a matching in $R'$ (there cannot be two edges through any $V_i^-$ since $V_i$ has outdegree $\leq 1$, and there cannot two edges through any $V_j^+$ since $V_j$ has indegree $\leq 1$). Also $e(M)=e(D)=2e(A)+e(B)=2f(E(R))=2\nu_f(R)$, and so $M$ satisfies the lemma.
\end{proof}
 
\subsection{The method of Besomi, Pavez-Sign\'e, and Stein}
Our tool for embedding trees in regularity partitions will be the following lemma of Besomi, Pavez-Sign\'e, and Stein. 
\begin{lemma}[Besomi, Pavez-Sign\'e, Stein, \cite{besomi2019degree}]\label{lem:stein_tree_embedding}
Let $M_0^{-1}, \Delta, \varepsilon\gg n^{-1}, k^{-1}$.
Let $T$ be a tree with $\Delta(T)\leq \Delta$ and parts of size $k_1, k_2$ having $k_1+k_2=k$.
Let $G$ be a graph and $R$ a $(\varepsilon, 5\sqrt{\varepsilon})$-reduced graph of $G$ with $|R|\leq M_0$. Suppose that $R$ is connected and satisfies one of:
\begin{enumerate}[(1)]
\item $R$ is bipartite with parts $X,Y$ and there is a subset $X'\subseteq X$ such that $|X'|\geq (1+100\sqrt{\varepsilon})k_1|R|/n$ and $\delta(X')\geq  (1+100\sqrt{\varepsilon})k_2|R|/n$. 
\item $R$ is non-biparite and has a matching with $\geq  (1+100\sqrt{\varepsilon})k|R|/n$ vertices.
\end{enumerate}
Then $G$ as a copy $T$.
\end{lemma}

We'll reframe this lemma in two ways in order to fit better with our approach of studying dense graphs via cut-density. 
The following reframing concerns part 2 of Lemma~\ref{lem:stein_tree_embedding}
\begin{lemma}\label{lem:regularity_tree_embedding_bipartite} 
Let $\Delta^{-1}, q, \gg \varepsilon\gg  k^{-1}$.
Let $T$ be a tree with $\Delta(T)\leq \Delta$ and parts of size $k_1, k_2$ having $k_1+k_2=k$.
Let $G$ be a  bipartite graph which is $q$-cut-dense with parts $X,Y$ such that $X$ contains $\geq k_1+400\sqrt{\varepsilon}k$  vertices of degree  $\geq  k_2+400\sqrt{\varepsilon}k$. Then $G$ contains a copy of $T$. 
\end{lemma}
\begin{proof}
Let $\Delta^{-1}, q, \gg \varepsilon\gg m^{-1}\gg M^{-1}\gg  k^{-1}$.
Let $n_0:=|G|$, noting $n_0=|X|+|Y|\geq |X|+\Delta(X)\geq k_1+400\sqrt{\varepsilon}k+k_2+400\sqrt{\varepsilon}k= k(1+800\sqrt{\varepsilon})$. If $k\leq qn_0/4$, then note that Fact~\ref{fact:cut_dense_min_degree} gives $\delta(G)\geq 0.9qn$, and so we can greedily find $T$ using Fact~\ref{fact:greedy_tree_embedding}. Thus suppose $k> qn_0/4$. 
Let $X'\subseteq X$ be the set of $\geq k_1+400\sqrt{\varepsilon}k$  vertices of degree  $\geq  k_2+400\sqrt{\varepsilon}k$. 
Use Szemer\'edi's Regularity Lemma to get a subgraph $G'$ of $G$ and a $(\varepsilon, 5\sqrt{\varepsilon})$-regularity partition with $s+t\in [m,M]$ parts  $X_1, \dots, X_s, Y_1, \dots, Y_t$  refining the partition $\{X', X\setminus X', Y\}$, labelled so that $X_1, \dots, X_r\subseteq X', X_{r+1}, \dots, X_s\subseteq X\setminus X', Y_1, \dots, Y_t\subseteq Y$. Let $n:=|G'|$, noting that $n\geq(1-\varepsilon)n_0\geq (1-\varepsilon)(1+800\sqrt{\varepsilon})k\geq k$, which in particular gives $\Delta^{-1}, q, \varepsilon, m^{-1}, M^{-1}\gg n^{-1}$.
Let $R$ be the reduced graph of this partition. By Fact~\ref{fact:cut_dense_preserved_by_regularity}, $G'$ is $(q-2(\varepsilon+5\sqrt{\varepsilon}))\geq q/2$-cut-dense, and so by Lemma~\ref{lem:cut_dense_connected_cluster_graph}, $R$ is connected. Note that all $X_iX_j$ and $Y_iY_j$ must be non-edges in this partition because there are no edges between these pairs in $G'$ (since there are none in $G$). 
We have that $s\geq (|X'|-\varepsilon n_0)|R|/n\geq  (k_1+400\sqrt{\varepsilon}k-\varepsilon n_0)|R|/n\overset{k> qn_0/4}{\geq}  (k_1+400\sqrt{\varepsilon}k-4\varepsilon k/q)|R|/n {\ge} (k_1+(400\sqrt{\varepsilon}-4\varepsilon /q) k_1 )|R|/n \geq (1+100\sqrt{\varepsilon})k_1|R|/n$.
For $i\leq s$, note that $|N_R(X_i)|(n/|R|)^2+|R\setminus N_R(X_i)|(5\sqrt {\varepsilon})(n/|R|)^2\geq e_{G'}(X_i, V(G))$. Using ``$d_G(v)-d_{G'}(v)\leq \varepsilon n$'', we also have $e_{G'}(X_i, V(G))= \sum_{v\in X_i}d_{G'}(v)\geq \sum_{v\in X_i}(d_G(v)-\varepsilon n)\geq \sum_{v\in X_i}( k_2+400\sqrt{\varepsilon}n-\varepsilon n)\geq k_2n/|R|+399\sqrt{\varepsilon}n^2/|R|$.
Combining gives $|N_R(X_i)|\geq \frac{k_2n/|R|+399\sqrt{\varepsilon}n^2/|R|- |R\setminus N_R(X_i)|(5\sqrt{\varepsilon})(n/|R|)^2}{(n/|R|)^2} = k_2|R|/n+399\sqrt{\varepsilon}|R|-|R\setminus N_R(X_i)|(5\sqrt{\varepsilon}) \geq  k_2|R|/n+(399\sqrt{\varepsilon}-5\sqrt{\varepsilon})|R|\overset{n\ge k_2}{\geq} (1+100\sqrt{\varepsilon})k_2|R|/n$.
Now, Lemma~\ref{lem:stein_tree_embedding} (1) applies to give a copy of $T$.
\end{proof}

The following is the second reframing of the Besomi, Pavez-Sign\'e, Stein lemma that we'll use. 
\begin{lemma}\label{lem:regularity_tree_embedding_nonbipartite} 
Let $\Delta, q, \delta\gg \varepsilon\gg   k^{-1}$.
Let $T$ be a tree with $\Delta(T)\leq \Delta$ and parts of size $k_1, k_2$ having $k_1+k_2=k$.
Let $G$ be a order $n_0$ graph which is $q$-cut-dense which has no copy of $T$. Then one can delete $\delta n_0^2$ edges from $G$ to get  graph $H$ satisfying one of:
\begin{enumerate}[(i)]
\item $H$ is bipartite and $q/2$-cut-dense.
\item There are subsets $C_1, C_2\subseteq V(G)$ with $2|C_1|+|C_2|\leq (1+100\sqrt{\varepsilon})k$ such that all edges of $H$ either touch $C_1$ or are contained in $C_2$.
\end{enumerate}
\end{lemma}
\begin{proof}
Pick $\Delta, q\gg \varepsilon\gg m^{-1}\gg M^{-1}\gg  k^{-1}$.
If $n_0\leq (1+100\sqrt{\varepsilon})k$, then setting $C_2=V(G)$ gives the structure for (ii), so suppose $n_0> (1+100\sqrt{\varepsilon})k$.
Apply the Szemeredi's Regularity Lemma in order to get a subgraph $G'$ and a $(\varepsilon/6, 4\sqrt{\varepsilon})$-regularity partition with $t\in [m,M]$ parts $V_1, \dots, V_t$. Denote by $R$ the corresponding $(\varepsilon/4, 4\sqrt{\varepsilon})$-reduced graph with $|R|=t$. From the remark after the statement of the regularity lemma, we have $e(G\setminus G')\leq (2\varepsilon/6+4\sqrt{\varepsilon})n_0^2\ll \delta n_0^2$. Let $n:=|V(G')|\in[(1-\varepsilon)n_0, n_0]$, noting that each part has $|V_i|=n/|R|$ and that  $n\geq(1-\varepsilon)n_0\geq (1-\varepsilon)(1+100\sqrt{\varepsilon})k\ge  k$, giving $\Delta, q\gg \varepsilon, m^{-1}, M^{-1}\gg n^{-1}$.   By Fact~\ref{fact:cut_dense_preserved_by_regularity}, $G'$ is $(q-2(4\sqrt{\varepsilon}+\varepsilon/4))\geq (q-9\sqrt{\varepsilon})\geq q/2$-cut-dense. By Lemma~\ref{lem:cut_dense_connected_cluster_graph}, $R$ is connected (since $q\gg 9\sqrt{\varepsilon}$).  

Suppose that $\nu_f(R)\geq (1+100\sqrt{\varepsilon})k|R|/2n$. Apply Lemma~\ref{lem:matching_from_fractional_matching} to get a $(\varepsilon/2, 2\sqrt{\varepsilon})$-reduced graph $R'$ of $G'$ which has a matching with $|V(M)|=2\nu(R')\geq 2|R'|\frac{\nu_f(R)}{|R|}\geq (1+100\sqrt{\varepsilon})k|R'|/n$. If $R'$ is bipartite, then so is $G'$ (if $V(R')=X\cup Y$ is a bipartition of $R'$, then $X':=\bigcup_{U\in X}U, Y'=\bigcup_{U\in Y}U$ gives a bipartition of $G'$) and so we are in (i). Otherwise, note that by Lemma~\ref{lem:cut_dense_connected_cluster_graph}, $R'$ is again connected, and so Lemma~\ref{lem:stein_tree_embedding} gives a copy of $T$.

Suppose that $R$  has $\nu_f(R)< (1+100\sqrt{\varepsilon})k|R|/2n$.  By Theorem~\ref{thm:fractional_matching}, there is a fractional cover $g:V(R)\to \{0, 1/2, 1\}$ with $g(V(R))=\tau_f(R)=\nu_f(R)\leq (1+100\sqrt{\varepsilon})k|R|/2n$. Let $D_1=g^{-1}(1)$ and $D_2=g^{-1}(1/2)$, noting that $|D_1|+|D_2|/2=g(V(G))\leq(1+100\sqrt{\varepsilon})k|R|/2n$. Also note that every edge $xy$ of $R$ must either touch $D_1$ or be contained in $D_2$ (by definition of fractional cover $g(x)+g(y)\geq 1$, meaning that either $g(x)=g(y)=1/2$ or at least one of $g(x)$/$g(y)$ equals $1$).
   Let $C_1=\bigcup_{V_i\in D_1} V_i$, $C_2=\bigcup_{V_j\in D_2} V_j$ be the corresponding subsets of $G'$. 
Since every edge $xy\in G'$  is contained in some pair $(V_i,V_j)$ with $V_iV_j\in E(R)$, we have that $xy$ either touches $C_1$ or is contained in $C_2$. Note that $2|C_1|+|C_2|=(2|D_1|+|D_2|)n/|R|\leq (1+100\sqrt{\varepsilon})k$.
\end{proof}

\subsection{Stability lemmas}
We'll need the following auxiliary lemma.
\begin{lemma}\label{lem:find_good_graph} 
Let $G$ be a bipartite graph with parts $X,Y$ having $|Y|\geq C^2k$ and $e(G)\geq k|Y|$. There is a subgraph $H$ with $|H\cap X|= k-|X|/C$,  $|H\cap Y|= C^2k$, and $\delta_H(H\cap X)\geq Ck/4$.
\end{lemma}
\begin{proof} 
Let $Y'\subseteq Y$ be a random subset of order $C^2k$, noting that $\mathbb E(e(X, Y'))=e(G)|Y'|/|Y|\geq k|Y'|$. Fix some set $Y'$ for which this happens. Let $X'=\{x\in X: |N(x)\cap Y'|\geq Ck/4\}$. Since $|X'|C^2k+ |X|Ck\geq |X'|C^2k+ (|X|-|X'|)Ck/4=|X'||Y'|+ (|X|-|X'|)Ck/4\geq e(X,Y')\geq k|Y'|=C^2k^2$, we get that $|X'|\geq k-|X|/C$.
\end{proof}

We now come to the first stability lemma that we'll use --- this one finds structure in $T$-free bipartite graphs that are sufficiently unbalanced.
\begin{lemma}\label{lem:stability_bipartite}
Let $1\gg \Delta^{-1}, \varepsilon, q\gg \gamma\gg  k^{-1}$.
Let $T$ be a $k$-edge tree with $\Delta(T)\leq \Delta$. 
Let $G$ be a bipartite graph with parts $X,Y$, $|Y|\geq 400\varepsilon^{-1/2}k$,   $\delta(G)\geq (1-\varepsilon)k/2$, and let $D$ be a $q$-cut-dense  bipartite graph with parts $X_D, Y_D$ having $X\cap V(G)\subseteq X_D$, $|X_D|\le 10k$. If $G\cup D$ has no copy of $T$, then there is a subgraph $H$ of $G$ satisfying:
\begin{itemize} 
\item $H$ is bipartite with parts $X',Y'$ such that $|X'|\leq (1+20\gamma)k/2$, $|Y'|\geq (1-8\varepsilon^{1/32})|Y|\geq \varepsilon^{-1/8}k/2$, $\delta(X')\geq (1-8\varepsilon^{1/32})|Y'|$, $\delta(Y')\geq (1-8\varepsilon^{1/32})k/2$.  
\end{itemize}
\end{lemma}
\begin{proof}
Let $\Delta, \varepsilon,q\gg \gamma\gg   k^{-1}$.  Let the parts of $T$ have size $k_1, k_2$ with $k_1\le k_2$. 
 If $|D|\geq 2q^{-1} k$, then by Fact~\ref{fact:cut_dense_min_degree}, $\delta(D)\geq (0.9 q)(2q^{-1} k)\ge 1.8k$ and Fact~\ref{fact:greedy_tree_embedding} gives a copy of $T$.  So assume  $|D|< 2q^{-1} k$.  
 Say a subgraph $H$ of $G$ is good if $|V(H)\cap Y|= 400\varepsilon^{-1/8} k$, $|V(H)\cap X|\geq \varepsilon^{1/8} k/2$, and  $\delta_H(X\cap H)\geq 2k$. Let $H_1, \dots, H_s$ be  a maximal collection of vertex-disjoint good subgraphs. Let  $H=\bigcup H_i$. 
 Let $J$ be the subgraph on $V(G)$ consisting of edges disjoint from $V(H)$. 

We claim that $e(J)\leq \varepsilon^{1/8} k|Y|$. Indeed otherwise apply Lemma~\ref{lem:find_good_graph} to $J$ with $X=X, Y=Y, k'= \varepsilon^{1/8} k, C'=20\varepsilon^{-1/8}$ in order to get a subgraph with $|V(H)\cap Y|= (20\varepsilon^{-1/8})^2\varepsilon^{1/8}k=400\varepsilon^{-1/8}k$, $|V(H)\cap X|\geq \varepsilon^{1/8} k- |X|/(20\varepsilon^{-1/8})\overset{{|X|\leq 10k}}{\geq} \varepsilon^{1/8} k/2$, and  $\delta_H(X\cap H)\geq (20\varepsilon^{-1/8})\varepsilon^{1/8} k/4\geq 2k$. This is a good subgraph, contradicting maximality.
 
Note that since $|X|\leq 10k$  and $H_1, \dots, H_s$ are vertex disjoint, each having $\varepsilon^{1/8} k/2$ vertices in $X$, we have $s\le |X|/(\varepsilon^{1/8} k/2)\le 10k/(\varepsilon^{1/8} k/2)= 20\varepsilon^{-1/8}$ and so $H:=H_1\cup \dots \cup H_s$ has $|H\cup D|\leq |D|+(400\varepsilon^{-1/8} k)s \leq 2q^{-1}k+ (20\varepsilon^{-1/8})(400\varepsilon^{-1/8}k)=  (2q^{-1}+8000\varepsilon^{-1/4})k$. Also vertices in $H\setminus D\subseteq H\setminus X\subseteq Y$ have $\delta(G)\geq (1-\varepsilon)k/2\geq q|D|/2$ neighbours in $D$. 
Note $|D|\geq |X|\geq \delta(G)\geq k/3$.
By  Lemma~\ref{lem:cut_dense_add_vertices}, we get that $H\cup D$ is $\frac{q(q/2)|D|^2}{4|H\cup D|^2}\geq \frac{q(q/2)(k/3)^2}{4(2q^{-1}+8000\varepsilon^{-1/4})^2k^2}\geq \varepsilon q^4$-cut-dense. 

 If $|H\cap X|\geq k_1+\gamma k$, then Lemma~\ref{lem:regularity_tree_embedding_bipartite} (applied to $H\cup D$ with $\varepsilon'=(\gamma/200)^2$) gives a copy of $T$, so we can assume that $|H\cap X|< k_1+\gamma k \leq k/2+10\gamma k=(1+20\gamma)k/2< (1+\varepsilon)k/2$. 
Let $Y'=\{y\in Y:|N(y)\cap H\cap X|\geq (1-4\varepsilon^{1/16})k/2 \}$. 
We have   $(|Y|-|Y'|-20\cdot 400\varepsilon^{-1/8}\cdot \varepsilon^{-1/8}k)\varepsilon^{1/16} k\leq (|Y|-|Y'|-s400\varepsilon^{-1/8}k)\varepsilon^{1/16} k = (|Y|-|Y'|-|Y\cap H|)\varepsilon^{1/16} k\le |Y\setminus (H\cup Y')|\varepsilon^{1/16} k = \sum_{y\in Y\setminus (H\cup Y')}((1-2\varepsilon^{1/16})k/2-(1-4\varepsilon^{1/16})k/2) )\leq \sum_{y\in Y\setminus (H\cup Y')} (\delta(G)-(1-4\varepsilon^{1/16})k/2)\leq \sum_{y\in Y\setminus (H\cup Y')} (|N(y)|-|N(y)\cap X\cap H|) = \sum_{y\in Y\setminus (H\cup Y')} |N(y)\cap X\setminus H)|= e(X\setminus H, Y\setminus (H\cup Y'))\leq e(J)\leq \varepsilon^{1/8}k|Y|$, which rearranges to $|Y'|\geq |Y|-20\cdot 400\varepsilon^{-1/4}k-\varepsilon^{1/16}|Y|\geq (1-\frac{20\cdot 400\varepsilon^{-1/4}k}{400\varepsilon^{-1/2}k}-\varepsilon^{1/16})|Y|\geq (1-2\varepsilon^{1/16})|Y|$.
 In particular, $Y'\neq \emptyset$, which shows $k_1+10\gamma k\geq |H\cap X|\geq (1-4\varepsilon^{1/16})k/2$, and so $k_1, k_2\geq (1-4\varepsilon^{1/16})k/2-10\gamma k\geq (1-5\varepsilon^{1/16})k/2$.   Let $\hat X=\{x\in H\cap X: |N(x)\cap Y'|\leq (1-6\varepsilon^{1/32})|Y'|\}$. Since $(1+20\gamma)|Y'|k/2-|\hat X||Y'|6\varepsilon^{1/32} \geq |H\cap X||Y'|-6\varepsilon^{1/32}|\hat X||Y'|\overset{\hat X\subseteq H\cap X}{=}(|\hat X|+|H\cap X\setminus \hat X|)|Y'|-6\varepsilon^{1/32}|\hat X||Y'|= |\hat X|(1-6\varepsilon^{1/32})|Y'|+ |H\cap X\setminus \hat X||Y'|\geq e(H\cap X, Y')\geq |Y'|(1-5\varepsilon^{1/16})k/2$ we have $|\hat X|\leq \frac{(1+20\gamma)|Y'|k/2-|Y'|(1-5\varepsilon^{1/16})k/2}{6\varepsilon^{1/32}|Y'|}= \frac{(5\varepsilon^{1/16}+20\gamma)k} {2\cdot 6\varepsilon^{1/32}}\leq \varepsilon^{1/32}k$. 
 
Now the lemma holds with $H=H\setminus \hat X$,  $X':=H\cap X\setminus \hat X$ and $Y'$: indeed, we have  $|X'|\le |H\cap X|\le (1+20\gamma )k/2$, $|Y'|\geq (1-2\varepsilon^{1/16})|Y|\ge (1-8\varepsilon^{1/32})|Y|\geq \varepsilon^{-1/8}k/2$, $\delta(X')\geq (1-6\varepsilon^{1/32})|Y'|\ge (1-8\varepsilon^{1/32})|Y'|$, and $\delta(Y')\geq (1-4\varepsilon^{1/16})k/2-|\hat X|\geq (1-4\varepsilon^{1/16})k/2-\varepsilon^{1/32}k\geq (1-8\varepsilon^{1/32})k/2$.
\end{proof}

The second stability lemma that we'll use finds structure in $T$-free bipartite graphs that are dense.
\begin{lemma}\label{lem:stability_nonbiparite} 
Let $\Delta^{-1}\gg  q, \varepsilon \gg  k^{-1}$ and $n\leq \varepsilon^{-1/4}k$.
Let $T$ be a tree with $\Delta(T)\leq \Delta$. 
Let $G$ be a $n$-vertex graph which is $q$-cut-dense,  has no copy of $T$, and has $e(G)\geq (1-\varepsilon)kn/2$. Then there is a subgraph $H$ of $G$ with $e(H)\geq e(G)(1- 16\varepsilon^{1/8})$ satisfying:
\begin{itemize}
\item $\delta(H)\geq (1-16\varepsilon^{1/8})k$ and $|H|\leq (1+16\varepsilon^{1/8})k$. 
\end{itemize}
\end{lemma}
\begin{proof}
Pick $\Delta^{-1}\gg  q, \varepsilon \gg \gamma \gg n^{-1}, k^{-1}$.
Suppose that $n\leq (1+4\sqrt{\varepsilon})k$.  Then $e(G)\geq  (1-\varepsilon)kn/2\ge  (1-\varepsilon)\frac{n}{(1+4\sqrt{\varepsilon})}n/2\geq (1-6\sqrt{\varepsilon})n^2/2\geq \binom n2-8\sqrt{\varepsilon}n^2$, there are $\leq 8\varepsilon^{1/4}n$ vertices with degree $\leq n-1-4\varepsilon^{1/4}n\le (1-8\varepsilon^{1/4})n$. Hence there are $\leq 8\varepsilon^{1/4}n(n-1)= 16\varepsilon^{1/4}\binom n2\leq 16\varepsilon^{1/4}e(G)$ edges through such vertices. Delete all such vertices in order to get a graph $H$.  We have $|H|\le n\le  (1+4\sqrt{\varepsilon})k\le (1+16\varepsilon^{1/8})k$ and $\delta (H)\ge (1-8\varepsilon^{1/4})n-|G\setminus H|\geq (1-8\varepsilon^{1/4})n-8\varepsilon^{1/4}n=(1-16\varepsilon^{1/8})k$ --- and so $H$ satisfies the lemma. 
Thus we can assume $n> (1+4\sqrt{\varepsilon})k$. 
Apply Lemma~\ref{lem:regularity_tree_embedding_nonbipartite} with $\varepsilon'=(\gamma/100)^2$ to get a subgraph $H$ with $e(G)-e(H)\leq \varepsilon n^2\leq\varepsilon e(G)$. 

Suppose that $H$ satisfies (ii). Then $e(H)\leq |C_1|(n-|C_1|)+\binom {C_1}2+\binom{C_2}2\leq  |C_1|(n-|C_1|)+|C_1|^2/2+|C_2|^2/2= (|C_1|+|C_2|/2)n-|C_1|^2/2-|C_2|(n-|C_2|)/2\le (1+100\sqrt{\varepsilon'})\frac k2 n-|C_1|^2/2-|C_2|(n-|C_2|)/2 = (1+ \gamma)kn/2-|C_1|^2/2-|C_2|(n-|C_2|)/2$. If $2\sqrt{\varepsilon} n\leq |C_2|\leq  (1-2\sqrt{\varepsilon})n$ or $|C_1|\geq 2\sqrt{\varepsilon} n$ or $|C_1|+2|C_2|\leq k-2\sqrt{\varepsilon} n$, then this is $\leq kn/2-2\varepsilon n^2$, contradicting  $e(G)\geq (1-\varepsilon)kn/2$. 

Thus we have $|C_1|\leq 2\sqrt{\varepsilon} n$, $|C_1|+2|C_2|\geq k- 2\sqrt{\varepsilon} n$, which imply that $|C_2|\geq k/2-\sqrt{\varepsilon}n-\sqrt{\varepsilon}n>\varepsilon^{1/4}n-2\sqrt{\varepsilon}n\geq 2\sqrt{\varepsilon}n$. Since we don't have $2\sqrt{\varepsilon}n\leq |C_2|\leq  (1-2\sqrt{\varepsilon})n$, this gives $|C_2|> (1-2\sqrt{\varepsilon})n$. Since $|C_2|\leq (1+\gamma)k$ (from Lemma~\ref{lem:regularity_tree_embedding_nonbipartite}), we get $n\leq (1+4\sqrt{\varepsilon})k$, contradicting the first paragraph.

Suppose that $H$ satisfies (i). 
Let the parts of $T$ have sizes $k_1\leq k_2$. 
We have that $H$ is bipartite with parts $X,Y$ --- suppose these are labelled so that $|X|\leq |Y|$. Note $k_1\geq k/2\Delta$ by Fact~\ref{fact:large_bipart_in_bounded_degree_tree}. Set $k_1':=(1+\gamma)k_1, k_2':=(1+\gamma)k_2, k':=k'_1+k'_2=(1+\gamma)k$. 
Let $X_1=\{x\in X: d_H(x)\geq k_2'\}$ and $X_2=\{x\in X: d_H(x)< k_2'\}$.
 From Lemma~\ref{lem:regularity_tree_embedding_bipartite} applied with $\varepsilon'=(\gamma/200)^2$, we have $|X_1|\leq k_1'$, which gives $|X_1|(n-|X_1|-|X_2|)\leq k_1'(n-k_1'-|X_2|)$ (this inequality rearranges to $(n-|X_2|-|X_1|-k_1)(k_1-|X_1|)\geq 0$. This holds because $k_1'\geq |X_1|$, $n/2> (1+4\sqrt{\varepsilon})k/2\geq (1+4\sqrt{\varepsilon})k_1\geq k_1'$, and  $n/2\ge (|Y|+|X|)/2\geq |X|= |X_1|+|X_2|$). 
We have $e(H)\leq |X_1||Y|+|X_2|k_2'=|X_1|(n-|X_1|-|X_2|)+|X_2|k_2'\leq k_1'(n-|X_2|-k_1')+|X_2|(k'-k_1')=nk'/2-(n-2|X_2|)(k'/2-k_1')-(k_1')^2\leq nk'/2-(k_1')^2\leq nk'/2-(k/2\Delta)^2\leq (1+\gamma)nk/2-\varepsilon^{1/4} nk/4\Delta^2$. This contradicts $e(G)\geq (1-\varepsilon)nk/2$.
\end{proof}

\begin{lemma}\label{lem:stability_nonbiparite_no_cut_dense} 
Let $\Delta^{-1}\gg  \varepsilon \gg  k^{-1}$ and $n\leq \varepsilon^{-1/4}k$.
Let $T$ be a tree with $\Delta(T)\leq \Delta$. 
Let $G$ be a $n$-vertex graph, that  has no copy of $T$, and has $e(G)\geq (1-\varepsilon)kn/2$. Then there is a subgraph $H$ of $G$ satisfying:
\begin{itemize}
\item $\delta(H)\geq (1-32\varepsilon^{1/8})k$ and $|H|\leq (1+32\varepsilon^{1/8})k$.
\end{itemize}
\end{lemma}
\begin{proof}
Let $\Delta^{-1}\gg  \varepsilon \gg q\gg  k^{-1}$
Apply Lemma~\ref{lem:cut_dense_decomposition} to delete $q n^2$ edges in order to get a graph $H$ whose connected components are $q$-cut-dense. We have $e(H)\geq (1-\varepsilon)kn/2-qn^2\geq (1-\varepsilon)kn/2-q\varepsilon^{-1/4}nk\ge  (1-2\varepsilon)kn/2$. By Fact~\ref{fact:connectected_component_with_same_density_as_G}, one of the connected components of $H$ has $e(C_i)\geq (1-2\varepsilon)k|C_i|/2$.
Now Lemma~\ref{lem:stability_nonbiparite} applies to $C_i$ with $\varepsilon'=2\varepsilon$, $q=q$ to give  a subgraph $H$ with  $\delta(H)\geq (1-16(2\varepsilon)^{1/8})k\ge(1-32\varepsilon^{1/8})k$ and $|H|\leq (1+16(2\varepsilon)^{1/8})k\le (1+32\varepsilon^{1/8})k$.
\end{proof}

Our previous two stability lemmas are combined in the following one.
\begin{lemma}\label{lem:stability_main_proof}
Let $\Delta\gg \varepsilon\gg  k^{-1}$.
Let $T$ be a $k$-vertex tree with $\Delta(T)\leq \Delta$.
Let $G$ be a graph with $e(G)\geq (1-\varepsilon)kn/2$, and let $D$ be a cover   of order $|D|\leq 10k$. If $G$ has no copy of $T$, then there is a subgraph $H$ of $G$ 
satisfying one of:
\begin{enumerate}[(i)]
\item $\delta(H)\geq (1-32\varepsilon^{1/1024})k$ and $|H|\leq (1+32\varepsilon^{1/1024})k$.
\item  $H$ is bipartite with parts $X,Y$ such that $|X|\leq (1+32\varepsilon^{1/1024})k/2$, $|Y|\geq \varepsilon^{-1/1024}k\geq 6k$, $\delta(X)\geq (1-32\varepsilon^{1/1024})|Y|$, $\delta(Y)\geq (1-32\varepsilon^{1/1024})k/2$.
\end{enumerate}
\end{lemma}
\begin{proof}
Let $\Delta\gg \varepsilon\gg p\gg \gamma\gg n^{-1}, k^{-1}$.
If $n\leq \varepsilon^{-1/8}k$, then Lemma~\ref{lem:stability_nonbiparite_no_cut_dense} applies to give structure (i), so we can suppose that $n< \varepsilon^{-1/8}k$.

Delete all $\leq 100k^2\leq 100\varepsilon^{1/8} kn$ edges inside $D$ to get a bipartite graph $G[D,Y]$. Apply Lemma~\ref{lem:cut_dense_decomposition_of_dominated_graphs2} to delete $200p kn\leq \varepsilon^{1/8} kn/2$ edges  in order to partition the remaining edges into  subgraphs $G_1, \dots, G_t$ such that each $G_i$ has a $p^{10^9p^{-8}}$-cut-dense subgraph $D_i$ with $|D_i|\leq  430p^{-8}k$ containing $X\cap G_i$. Note $e(\bigcup G_i)\geq e(G)-101\varepsilon^{1/8} kn\le (1-\varepsilon)kn/2-101\varepsilon^{1/8} kn\ge (1-400\varepsilon^{1/8})kn/2$. By Fact~\ref{fact:connectected_component_with_same_density_as_G}, some $G_i$ has $e(G_i)\geq (1-400\varepsilon^{1/8})k|G_i|/2$. 
Using Fact~\ref{fact:subgraph_min_degree_e/2}, we find a subgraph $G_i'\subseteq G_i$ with  $\delta(G_i')\geq (1-400\varepsilon^{1/8})k/2\ge  (1-\varepsilon^{1/10})k/2$ and $e(G_i')\geq (1-400\varepsilon^{1/8})k|G_i'|/2\ge  (1-\varepsilon^{1/10})k|G_i'|/2$.

If $|G_i'|\leq (\varepsilon^{1/10})^{-1/4}k$, then Lemma~\ref{lem:stability_nonbiparite_no_cut_dense} applies with $\varepsilon'=\varepsilon^{1/10}$ to give structure (i).
If $|G_i'|> (\varepsilon^{1/10})^{-1/4}k\ge 400(\varepsilon^{1/32})^{-1/2}$, then Lemma~\ref{lem:stability_bipartite} applies with  $\varepsilon'=\varepsilon^{1/32}$, $q=p^{10^9p^{-8}}$, $\gamma =\gamma$  to give structure (ii).\end{proof}

We can now deduce Theorem~\ref{thm:stability_main}.
\begin{proof}[Proof of Theorem~\ref{thm:stability_main}]
Apply Lemma~\ref{lem:stability_main_proof} to get a subgraph $H$ with structure (i) or (ii). In case (i), we have $\delta(H)\geq (1-32\varepsilon^{1/320})k\ge (1-\alpha)k$ and $|H|\leq (1+32\varepsilon^{1/320})k\le (1+\alpha) k$ so case (i) of Theorem~\ref{thm:stability_main} holds. In case (ii), $|X|\leq (1+32\varepsilon^{1/320})k/2\le (1+\alpha)k/2$, $|Y|\geq \varepsilon^{-1/320}k\geq 6k$, $\delta(X)\geq (1-32\varepsilon^{1/320})|Y|\ge (1-\alpha)|Y|$, $\delta(Y)\geq (1-32\varepsilon^{1/320})k/2\ge (1-\alpha)k/2$, so case (ii) of Theorem~\ref{thm:stability_main} holds.
\end{proof}

\section{Extremal case analysis}
In this section, we prove Theorems~\ref{thm:nonbipartite_extremal} and~\ref{thm:bipartite_extremal}. 
For both of our extremal case analyses, we'll need a standard fact about dividing a tree in two.

\begin{lemma}[Proposition 3.22, \cite{montgomery2019spanning}]\label{lem:divide_tree}
Let $T$ be a tree and $m\le |T|/3$. 
There are two edge-disjoint subtrees $S$ and $R$ with $E(S)\cup S(R)=E(T)$ such that $|S|\in [m, 3m]$.
\end{lemma}

We'll need a variant of this where the tree is divided along an edge rather than a vertex.
\begin{lemma}\label{lem:split_tree_by_edge} 
Let $\alpha < 1/4\Delta$ and $k\geq 100$. 
Let $T$ be a $k$-edge tree with $\Delta(T)\leq \Delta$, and $\alpha \in [0,1]$. There is an edge $xy\in E(T)$, such that the component $T_y$ of $T-xy$ containing $y$ has $|T_y|\in [\alpha k, 3\Delta\alpha k]$.
\end{lemma}
\begin{proof}
Use Lemma~\ref{lem:divide_tree} to get subtrees $S,R$ dividing $E(T)$ with $|S|\in [\alpha\Delta k, 3\Delta \alpha k]$. Note that there is precisely one vertex $v$ in both $S$ and $R$. Let $y_1, \dots, y_d$ be the neighbours of $v$ in $S$, and let $S_{y_i}$ be the subtree of $S$ descending from each $y_i$. We have $d\le \Delta-1$ and $|S|=1+|S_{y_1}|+\dots, +|S_{y_d}|$, which gives that $|S_{y_i}|\ge |S|/\Delta\ge \alpha k$ for some $y_i$. This $S_{y_i}$ together with the edge $vy_i$ satisfy the lemma.
\end{proof}

\subsection{Non-bipartite}
The following is our analysis of near-extremal graphs that are non-bipartite. We restate Theorem~\ref{thm:nonbipartite_extremal} and prove it.
\begin{theorem*}\label{lem:extremal_analysis_nonbipartite_proof}
Let $\Delta^{-7}/1000\geq \varepsilon\ge 3d^{-1}$.
Let $G$ be a connected graph with $|G|\geq (1+256\Delta^2\sqrt{\varepsilon})d$, $\delta(G)\geq 128\Delta^2\sqrt{\varepsilon} d$ containing a subgraph $K$ with $|K|\leq (1+\varepsilon)d$, $\delta(K)\geq (1-\varepsilon)d$. Then $G$ contains a copy of every $d$-edge tree having $\Delta(T)\leq \Delta$.
\end{theorem*}
\begin{proof}[Proof of Theorem~\ref{thm:nonbipartite_extremal}]
Use Lemma~\ref{lem:split_tree_by_edge} to divide the tree $T$ into subtrees $S$ and $R$ having $|S|\in [4\Delta\varepsilon d, 12\Delta^2\varepsilon d]$ and an edge $st$ with $s\in S, r\in R$. Let the parts of $S$ have sizes $s_1, s_2$ with $s$ contained in the one of size $s_1$. Using Fact~\ref{fact:large_bipart_in_bounded_degree_tree}, note that $s_1,s_2\geq |S|/2\Delta\geq 2\varepsilon d$.
Let $A=\{v\not\in K: |N(v)\cap K|\geq 64\Delta^2{\sqrt{\varepsilon}} d\}$. There are two cases based on how big $A$ is.

If $|A|\geq 64\Delta^2 \sqrt{\varepsilon} d$, then fix a subset  $A'\subseteq A$ with $|A'|= 64\Delta^2 \sqrt{\varepsilon} d$. Note that $e(G[A',K])\ge (64\Delta^2\sqrt{\varepsilon} d)^2\ge (32\Delta^2 \varepsilon)(2d)^2\ge (32\Delta^2 \varepsilon)((1+\varepsilon )d+ 64\Delta^2 \sqrt{\varepsilon} d)^2\ge (32\Delta^2 \varepsilon d)|G[A',K]|$. Fact~\ref{fact:subgraph_min_degree_e/2} gives a subgraph $H\subseteq G[A', K]$ with $\delta(H)\geq 13\Delta^2\varepsilon d$. Get an embedding $f$ of $S\cup\{sr\}$ into $H$ using Fact~\ref{fact:greedy_tree_embedding} with $r$ embedded in $K$. Note that $\delta(K\setminus (im_f\setminus \{f(r)\}))\geq \delta(K)-|f(S)\cap K|\geq (1-\varepsilon)d-|f(S)\cap K|= (1-\varepsilon)d-s_2=|T|-s_2-\varepsilon d-1=(|R|+s_1+s_2)-s_2-\varepsilon d-1=|R|+s_1-1-\varepsilon d>|R|$. 
Now embed $R$ disjointly from the copy of $S$ and rooted at $f(r)$ using Fact~\ref{fact:greedy_tree_embedding}.

If $|A|\leq 64\Delta^2\sqrt{\varepsilon} d$, then  $H:=G\setminus (V(K)\cup A)$ has $|H|\geq |G|-|K|-|A|\geq (1+256\Delta^2\sqrt{\varepsilon})d-(1+\varepsilon)d-64\Delta^2\sqrt{\varepsilon} d>0$ and so $H$  is non-empty. Also $H$ has minimum degree $\delta(G)- 64\Delta^2\sqrt{\varepsilon} d\ge 128\Delta^2\sqrt{\varepsilon} d -64\Delta^2\sqrt{\varepsilon} d\geq 64\Delta^2\sqrt{\varepsilon} d>|S|$. By connectedness, there is an edge $r's'$ from $r'\in K\cup A$ to $s'\in H$. Use Fact~\ref{fact:greedy_tree_embedding} to embed $S$ to $H$ mapping $s$ to $s'$.  Use Fact~\ref{fact:greedy_tree_embedding} to  embed $R$ to $K\cup \{r'\}$  mapping $r$ to $r'$. Combining the two embeddings gives an embedding of $T$ into $G$.
\end{proof}

\subsection{Bipartite}

 To carry out our bipartite extremal case analysis we'll need the following auxiliary lemma about embedding trees. For a set of vertices $S$ in a tree $T$, we'll use $N^2_T(S)$ to denote the set of vertices in $T$ at distance $2$ from everything in $S$ i.e. $N^2(S)=N(N(S))\setminus (S\cup N(S))$.
\begin{lemma}\label{lem:greedy_embedding_bipartite} 
Let $\Delta^{-7}/1000\geq \varepsilon\ge k^{-1}$.
Let $G$ be a graph containing a bipartite graph $K$ with parts $X_K, Y_K$   such that $|X_K|\leq (1+3\varepsilon)k/2$, $\delta_{Y_K}(X_K)\geq 90\Delta^3\varepsilon k$, $\delta_{X_K}(Y_K)\geq (1-160\Delta^3\varepsilon) k/2$, and there are subsets $X_K^{big}\subseteq X_K, Y_K^{big}\subseteq Y_K$ with $|X_K\setminus X_K^{big}|\le \varepsilon k$, $|Y_K^{big}|\geq 4k$ such that $\delta_{Y_K^{big}}(X_K^{big})\ge  (1-20\Delta\varepsilon)|Y_K^{big}|$. Let $T$ be a $k$-edge tree with parts $X_T, Y_T$ having   $\Delta(T)\leq \Delta$. Suppose we have a subset $S\subseteq V(T)$ and an embedding $f:T[S\cup N_T(S)]\to G$ satisfying:
\begin{enumerate}
\item $|S|\leq 32\varepsilon\Delta^2 k$.
\item  $N_T(S)\subseteq Y_T$. 
\item $f(N_T(S))\subseteq Y_K$.
\item $|X_K\setminus f(S))|\geq |X_T\setminus S|$.
\end{enumerate}
\end{lemma}
\begin{proof}
Let $X_{K}^{small}= X_K\setminus X_K^{big}$, noting that we have $|X_K^{small}|\leq \varepsilon k$. 
Pick a set $R\subseteq X_T\setminus S$ with $|R|=|X_K^{small}\setminus im_f|$ such that vertices in $R$ are at distance $\geq 4$ from each other and from $S$ (pick vertices in $R$ one by one. Since $\Delta(T)\leq \Delta$, at any point there are $\le |R\cup S|\Delta^4\le (\varepsilon k+32\varepsilon\Delta^2k)\Delta^4\le 33\Delta^6 \varepsilon k$ vertices we can't pick. But using Fact~\ref{fact:large_bipart_in_bounded_degree_tree}, $|X_T\setminus S|\geq k/2\Delta-32\varepsilon\Delta^2k>33\Delta^6 \varepsilon k$).
Extend $f$ to map $R$ to $X_K^{small}\setminus im_f$ bijectively. Note that this gives an embedding since there are no edges from $R$ to $R\cup S\cup N_T(S)$. 
Note that $|im_f|=|R\cup S\cup N_T(S)|\le \varepsilon k+ 32\varepsilon\Delta^2k+\Delta ( 32\varepsilon\Delta^2k)\le  80\varepsilon\Delta^3k$.

Note that for all $r\in R$ and $y\in N_T(r)$, there are $\geq d_K(f(r))-|im_f|\geq \delta_{Y_K}(X_K)-|im_f|\ge 90\Delta^3\varepsilon k-80\varepsilon\Delta^3k> \Delta \varepsilon k\ge \Delta|R|\geq |N_T(R)|$ places where $y$ can be embedded in $Y_K$. By picking $f(y)$ for all such $y, r$ one by one, we can extend $f$ to an embedding of $R\cup N_T(R)\cup S\cup N_T(S)$ having $f(N_T(R))\subseteq Y_K$. 
Note that $|im_f|=|R\cup N_T(R)\cup S\cup N_T(S)|\le \varepsilon k+\Delta \varepsilon k+ 32\varepsilon\Delta^2k+\Delta ( 32\varepsilon\Delta^2k)\le  80\varepsilon\Delta^3k$.

Next, for all edges $yx\in E(T)$ with $y\in N_T(R)\cup N_T(S)$, and $x\in N^2_T(R)\cup N^2_T(S)$, note that there are $\geq d_{X_K}(f(y))-|im_f|\ge  \delta_{X_K}(Y_K)-|im_f|\geq (1-160\Delta^3\varepsilon) k/2- 80\varepsilon\Delta^3k>(1-200\Delta^4\varepsilon)k/2\ge\Delta^2(\varepsilon k+32\varepsilon\Delta^2 k)\ge \Delta^2|R|+ \Delta^2|S|\ge |N_T^2(R)\cup N_T^2(S))|$ places in $X_K$ where $x$ can be embedded. By picking $f(x)$ for all such $x, y$ one by one, we can extend $f$ to an embedding of $R\cup N_T(R)\cup N_T^2(R)\cup S\cup N_T(S)\cup N_T^2(S)$ having $f\big(N_T^2(R)\cup N_T^2(S)\big)\subseteq X_K\setminus X_K^{small}=X_K^{big}$ (since  $X_K^{small}\subseteq  f(R\cup S\cup N(S))$).  

Next, extend $f$ to map all vertices of $X_T\setminus R\cup N_T(R)\cup N_T^2(R)\cup S\cup N_T(S)\cup N_T^2(S)$ into $X_K\setminus im_f$.
This is possible to do injectively since $|X_T\setminus R\cup N_T(R)\cup N_T^2(R)\cup S\cup N_T(S)\cup N_T^2(S)|=|X_T\setminus S|-|R|-|N_T^2(R)|-|N_T^2(S)|$, $|X_K\setminus im_f|=|X_K\setminus f(S)|-|f(R)|-|f(N_T^2(R))|-|f(N_T^2(S))|$,  and  $|X_K\setminus f(S)|\geq |X_T\setminus S|$. Note that $f$ is an embedding since $T$ has no edges between $X_T\setminus R\cup N_T(R)\cup N_T^2(R)\cup S\cup N_T(S)\cup N_T^2(S)$ and $R\cup N_T(R)\cup N_T^2(R)\cup S\cup N_T(S)\cup N_T^2(S)$. 
 Note also that $|im_f\cap Y_K|\le 80\varepsilon\Delta^3k$ since no new vertices have been embedded in $Y_K$ in the last two paragraphs.

Now all unembedded vertices are in $Y_T$, and their neighbours are embedded to $X_K^{big}$. For such a vertex $y$, the number of places in $Y_K^{big}$ where $y$ can be embedded is at least:
\begin{align*}
\left|\bigcap_{x\in N_T(y)} N_{Y_K^{big}}(f(x))\setminus (im_f\cap Y_K^{big})\right|&\geq \left|\bigcap_{x\in N_T(y)} N_{Y_K^{big}}(f(x))\right|-|im_f\cap Y_K^{big}|\\
&\geq |Y_K^{big}|-\sum_{x\in N_T(y)}|Y_K^{big}\setminus N_{Y_K^{big}}(f(x))|-80\varepsilon\Delta^3k\\
&=|Y_K^{big}|-\sum_{x\in N_T(y)}(|Y_K^{big}|-|N_{Y_K^{big}}(f(x))|)-80\varepsilon\Delta^3k\\
&\ge   |Y_K^{big}|-\sum_{x\in N_T(y)}(|Y_K^{big}|-\delta_{Y_K^{big}}(X_K^{big}))-80\varepsilon\Delta^3k\\
&\geq|Y_K^{big}|-\sum_{x\in N_T(y)}(|Y_K^{big}|-(1-20\Delta\varepsilon)|Y_K|)-80\varepsilon\Delta^3k\\
&=    (1 -|N_T(y)| \cdot 20\Delta \varepsilon) |Y_K^{big}|-80\varepsilon\Delta^3k\\ 
&\ge   (1 -\Delta \cdot 20\Delta \varepsilon) 4k-80\varepsilon\Delta^3k \\
&= 3k+(1-80\Delta^2\varepsilon-80\varepsilon\Delta^3)k>2k
\end{align*}
 Picking $f(y)$ for such $y$ one by one we can complete the embedding of $T$.
\end{proof}

 We'll also need the following standard lemma about finding matchings/stars in graphs of a certain minimum degree. 
\begin{lemma}\label{lem:min_degree_star/matching} 
Let $G$ be a graph with $\delta(G)\geq d$ and $|G|=n\geq 10\Delta d$. Then $G$ either contains $d$ vertex-disjoint $\Delta$-stars or a matching of size $5d$.
\end{lemma}
\begin{proof}
Let $S_1, \dots, S_k$ be a collection of vertex-disjoint stars picked so that
\begin{enumerate}[(1)]
\item $1\leq e(S_i)\leq 2\Delta$ for all $i$.
\item $k$ is as large as possible, subject to (1) holding.
\item $e(\bigcup S_i)$ is as large as possible, subject to (1) and (2) holding.
\end{enumerate}
Suppose that the stars are ordered so that   $2\Delta \geq e(S_1)\geq e(S_2)\geq \dots \geq e(S_k)$. 
Suppose that $e(S_1), \dots, e(S_t)=2\Delta$, and $e(S_{t+1}), \dots, e(S_k)< 2\Delta$. If $t\geq d$ or $k\geq 5d$, we are done, so suppose $t< d$ and $k< 5d$. Using $|G|\geq 10\Delta d$, there is a vertex $v\not\in \bigcup S_i$. For $i=1, \dots, t$, let $c_i$ be the center of $S_i$ (which exists since $e(S_i)=2\Delta\geq 2$).  Note that $v$ has a neighbour $u$ outside $\{c_1, \dots, c_t\}$ (since $d(v)\geq \delta(G)\geq d>t$). 

If $u$ is a leaf of some star $S_i$ with $e(S_i)\geq 2$, then let $S_i'=S_i-u$, and $S_{k+1}'= \{uv\}$. Replacing $S_i$ with $S_i'$ and $S_{k+1}'$ gives a family of stars with (1) satisfied and larger $k$ (contradicting maximality in (2)).

If $u$ is the center of some star $S_i$, or contained in a single-edge star $S_i$, then let $S_i'=S_i+uv$. Replacing $S_i$ by $S_i'$ gives a family of stars with $e(\bigcup S_i)$ increased. Note that since $u\not\in \{c_1, \dots, c_t\}$, we have $e(S_i)<2\Delta$ which gives $e(S'_i)=e(S_i)+1\leq 2\Delta$, and so (1) is maintained. The number of stars didn't change, so (2) is maintained. This contradicts maximality in (3).
\end{proof}

We'll also need the fact that trees always either have a lot of leaves or a lot of bare paths. This was originally proved by Krivelevich in~\cite{krivelevich2010embedding}. The version below is taken from~\cite{montgomery2019spanning}, Lemma 2.1.
\begin{lemma}[\cite{krivelevich2010embedding}, Lemma 2.1]\label{Lemma_tree_leaves_bare_paths}
Every tree either has  $|V(T)|/10t$ leaves (and hence a matching of $|V(T)|/10\Delta(T)t$ leaves), or has $|V(T)|/10t$ disjoint bare paths of length $t$.
\end{lemma}

Now we can analyse the structure of near-extremal bipartite graphs. We restate Theorem~\ref{thm:bipartite_extremal} and prove it.
\begin{theorem*}\label{lem:extremal_analysis_bipartite_proof}
Let $\Delta^{-7}/1000\geq \varepsilon\ge k^{-1}$.
Let $T$ be a $k$-edge tree with $\Delta(T)\leq \Delta$.
Let $G$ be a connected graph with $\delta(G)\geq k/2$. Suppose that $G$ contains a bipartite subgraph $B$ with parts $X, Y$ such that $|X|\leq (1+\varepsilon)k/2$, $|Y|\geq 6k$, $\delta_B(X)\geq (1-\varepsilon)|Y|$, $\delta_B(Y)\geq (1-\varepsilon)k/2$. 
Then $G$ contains a copy of $T$.
\end{theorem*}
\begin{proof}[Proof of Theorem~\ref{thm:bipartite_extremal}]
Note that $|X|\geq \delta_B(Y)\geq (1-\varepsilon)k/2$. 
Let $Y'=\{y\in G\setminus X: |N_G(y)\cap X|\geq (1-160\Delta^3\varepsilon)k/2\}$, noting $Y\subseteq Y'$. 
Let $X'\subseteq \{x\in G\setminus X\cup Y': |N(x)\cap Y'|\geq 90\Delta^3\varepsilon k\}$ be a set of order $\min(|\{x\in G\setminus X\cup Y': |N(x)\cap Y'|\geq 90\Delta^3\varepsilon k\}|, \varepsilon k)$. Note that $G[X\cup X', Y]$ satisfies the assumptions of Lemma~\ref{lem:greedy_embedding_bipartite} on $K$ (with $X_K:= X\cup X', Y_K:= Y'$, $X_K^{big}=X, Y_K^{big}=Y$). Also $T$ satisfies the assumptions of Lemma~\ref{lem:greedy_embedding_bipartite} on $T$. 
Let $k_1\leq k_2$ be the sizes of the parts $X_T, Y_T$ of $T$ with $|X_T|=k_1$, $|Y_T|=k_2$, noting that $k_1+k_2=k+1$ and so $k_1\le \lfloor\frac{k+1}{2}\rfloor$.

Suppose  $|X\cup X'|\geq k_1$. Then  $S=\emptyset$, and the trivial function $f:\emptyset\to V(G)$ satisfy the assumptions of Lemma~\ref{lem:greedy_embedding_bipartite}. This gives a copy of $T$. 
Thus we can assume that $|X\cup X'|<k_1$. This gives $|X\cup X'|< k_1\le k_1\le \lfloor\frac{k+1}{2}\rfloor \overset{\text{F. \ref{fact:ceilfloor}}}{=}\lceil k/2\rceil$. Using that $|X\cup X'|$ is an integer, this implies $|X\cup X'|\le k/2$. Combined with $|X|\geq \delta(Y)\geq (1-\varepsilon)k/2$, this  tells us $|X'|<\varepsilon k$ and so $X'= \{x\in G\setminus X\cup Y': |N(x)\cap Y'|\geq 90\Delta^3\varepsilon k\}$.

Let $Z=V(G)\setminus (X\cup X'\cup Y')$, noting that vertices $z\in Z$ have $|N(z)\cap Z|\geq k/2-  |N(z)\cap X|-  |N(z)\cap (X'\setminus X)|- |N(z)\cap Y'|\geq k/2- (1-160\Delta^3\varepsilon)k/2- \varepsilon k- 90\Delta^3 \varepsilon k\geq 20\Delta^3\varepsilon k$ (using $|N(z)\cap X|\le(1-160\Delta^3\varepsilon)k/2$ since $z\not\in Y'$,  $|N(z)\cap (X'\setminus X)|\le |X'|\le \varepsilon k$,  $|N(z)\cap Y'|\le 90\Delta^3 \varepsilon k$ since $z\not\in X'$).
We break into cases.
\begin{itemize}
\item [Case 1:]
Suppose that $Z\neq \emptyset$. By connectedness, there is some edge $zv$ with $z\in Z, v\in X\cup X'\cup Y'$. Use Lemma~\ref{lem:split_tree_by_edge}  to get an edge $z_Tv_T$, such that the components $S,R$ of $T-z_Tv_T$ with $z_T\in S$, $v_T\in R$ have  $e(S)\in [4\Delta\varepsilon k, 16\Delta^2\varepsilon k]$.    Use Fact~\ref{fact:greedy_tree_embedding} to get an embedding $f:S\to Z$, embedding $z_T$ to $z$. 
Extend $f$ to embed $v_T$ to $v$. There are two subcases depending on whether $v\in X\cup X'$ or $v\in Y'$:
\begin{itemize}
\item[ Case 1.1:] Suppose $v\in Y'$. 
Let $Y_T'\in \{X_T, Y_T\}$ be the part of $T$  containing $v_T$, and $X_T'$ the other part of $T$. Since the parts of $S$ have size $\geq |S|/2\Delta\geq 2\varepsilon k$ (by Fact~\ref{fact:large_bipart_in_bounded_degree_tree}), we have $|X_T'\setminus S|\leq \max(|X_T|, |Y_T|)-2\varepsilon k\leq (1+\varepsilon)k/2-2\varepsilon k<| X|\le|X'\cup X|=|X_K\setminus im_f|$.
Thus  $S$, $f$, $X_T', Y_T'$ satisfy Lemma~\ref{lem:greedy_embedding_bipartite}, which gives an embedding of $T$. Indeed $|S|\le 16\Delta^2\varepsilon k\le 32\Delta^2\varepsilon k$ by construction, $N_T(S)=\{v_T\}\subseteq Y_T'$, $f(N_T(S))=\{v\}\subseteq Y'$, and $|X_K\setminus f(S)|=|X'\cup X|>|X_T'\setminus S|$. 

\item[Case 1.2:] Suppose $v\in X\cup X'$. Let $X_T'\in \{X_T, Y_T\}$ be the part of $T$  containing $v_T$, and $Y_T'$ the other part of $T$.
 Extend $f$ to embed $N_T(v_T)\setminus \{z_T\}$ to $N(v)\cap Y'$ (this is possible because all vertices in $X\cup X'$ have $\ge \min(\delta_{Y}(X), \delta_{Y'}(X'))\ge \min((1-\varepsilon)|Y|, 90\Delta^3\varepsilon k)= 90\Delta^3\varepsilon k>\Delta$ neighbours in $Y'$. Also nothing has been embedded into $Y'$ yet so there is nothing to avoid). Set $S'=S\cup \{v_T\}$. Since the parts of $S'$ have size $\geq |S'|/2\Delta\geq 2\varepsilon k$ (by Fact~\ref{fact:large_bipart_in_bounded_degree_tree}), we have $|X_T'\setminus S'|\leq \max(|X_T|, |Y_T|)-2\varepsilon k\leq (1+\varepsilon)k/2-2\varepsilon k< |X|-1\le |X'\cup X|-1=|X_K\setminus \{v\}|=|X_K\setminus im_f|$.
 Thus  $S'$, $f$, $X_T', Y_T'$ satisfy Lemma~\ref{lem:greedy_embedding_bipartite}, which gives an embedding of $T$. Indeed $|S'|\le 16\Delta^2\varepsilon k+1\le 32\Delta^2\varepsilon k$ by construction, $N_T(S')=N_T(v_T)\setminus \{z_T\}\subseteq Y_T'$ by ``$v_T\in X_T'$'', $f(N_T(S'))=f(N_T(v_T)\setminus \{z_T\})\subseteq Y'$ by construction, and $|X_K\setminus f(S')|=|X'\cup X|-1> |X_T'\setminus S'|$. 
\end{itemize}

\item [Case 2:]
Suppose that $Z=\emptyset$. Setting $d=\lceil k/2\rceil-|X'\cup X|$, recall that we've established $d> 0$. Since $|X|\geq (1-\varepsilon)k/2$, we have $d\le 2\varepsilon k$.
Since $Z=\emptyset$, we have $\delta(G[Y'])\geq \delta(G)-|V(G)\setminus Y'|=\delta(G)-|X'\cup X|\ge k/2-|X'\cup X|= d$. 
Using $k_1>|X'\cup X|$ we also get $d\ge k/2-|X'\cup X|= \frac{k_2+k_1-1}{2}-|X'\cup X|> \frac{k_2+k_1-1}{2}-k_1=\frac{k_2-k_1-1}{2}$. 
Since $|Y'|\ge |Y|\ge 6k\ge 10\Delta (2\varepsilon k)\ge 10\Delta d$, Lemma~\ref{lem:min_degree_star/matching} applies to $G[Y']$ to give one of the following cases:

\begin{itemize}
\item [Case 2.1:]
Suppose $G[Y']$ contains $d$ vertex-disjoint $\Delta$-stars $S'_1, \dots, S'_d$.  
Let $x_1, \dots, x_{d}\in X_T$ be vertices at pairwise distance $\geq 4$ in $T$ (they can be found since Fact~\ref{fact:large_bipart_in_bounded_degree_tree} gives $|X_T|\ge k/2\Delta\geq \Delta^4 (2\varepsilon k)\geq \Delta^4d$ ). Let $S_1, \dots, S_{d}$ be the stars in $T$ centered at these vertices, noting that they are vertex-disjoint and have no edges between them. Define $f:\bigcup S_i \to G$ to embed $S_i$ into $S_i'$.  
Now $S:=\{x_1, \dots x_d\}$ and $f$ satisfy Lemma~\ref{lem:greedy_embedding_bipartite}: Indeed $|S|= d\le 2\varepsilon k\le 32\varepsilon\Delta^2 k$, $N_T(S)\subseteq Y_T$ since $S\subseteq X_T$, $f(N_T(S))\subseteq im_f\subseteq Y'$ by construction, and $|X_K\setminus f(N_T(S))|=|X_K|=|X\cup X'|=\lfloor k/2\rfloor-d\ge |X_T|-d=|X_T\setminus S|$.  Invoking Lemma~\ref{lem:greedy_embedding_bipartite} gives an embedding of $T$.

\item [Case 2.2:]
Suppose $G[Y']$ contains a matching $M$ of size $3d$. 
By Lemma~\ref{Lemma_tree_leaves_bare_paths}, $T$ either contains $k/100\Delta$ disjoint bare paths   of length $10$, or contains a matching of $k/100\Delta $ leaves. We split into further cases.
\begin{itemize}
\item[Case 2.2.1:]  Suppose that $T$ contains $k/100\Delta\ge 2\varepsilon k\ge d$ disjoint bare paths  of length $10$. This implies that there are bare paths $P_1, \dots, P_d$ of length $4$ which are at distance $\geq 2$ from each other and whose endpoints are in $Y_T$ (take appropriate middle paths of length $4$ in the paths of length $10$ that we're given). Label $P_i=a_ib_ic_id_ie_i$ (so $a_i, c_i, e_i\in Y_T, b_i, d_i\in X_T$).  Take a submatching of $M'\subseteq M$ off size $2d$, and label it $M'=\{x_1^-y_1^-,x_1^+y_1^+,\dots, x_d^-y_d^-,x_d^+y_d^+\}$. 
Note that for all $y_1,y_2\in Y'$,  $|N(y_1)\cap N(y_2)\cap X|\ge |X|-|X\setminus N(y_1)|-|X\setminus N(y_2)|=|X\cap N(y_1)|+|X\cap N(y_2)|-|X|\ge (1-160\Delta^3\varepsilon)k/2+(1-160\Delta^3\varepsilon)k/2-k/2 =(1-320\Delta^3 \varepsilon )k/2\ge 2\varepsilon k\ge d$. 
Using this, for each $i=1, \dots, d$, we can pick a distinct common neighbour $z_i\in N(x_i^-)\cap N(x_i^+)\cap X$.  
Define  $f:\bigcup P_i\to G$ to embed each $P_i=a_ib_ic_id_ie_i$ to the path $y_i^-x_i^-z_ix_i^+y_i^+$.  
Now taking  $S=\{b_i,c_i,d_i: i=1, \dots, d\}$ together with this $f$ will satisfy Lemma~\ref{lem:greedy_embedding_bipartite}. Indeed $|S|=3d\le 3(2\varepsilon k)\le 32 \varepsilon\Delta^2k$, $N_T(S)=\{a_i,e_i:i=1, \dots, d\}\subseteq Y_T$, $f(N_T(S))\subseteq \{y_i^-, y_i^+: i=1, \dots, d\}\subseteq Y'$ by construction, and $|X_K\setminus f(S)|=|X_K\setminus \{z_1, \dots, z_d\}|=|X_K|-d=|X\cup X'|-d=\lceil k/2\rceil-2d\geq|X_T|-2d=  |X_T\setminus \{b_i, d_i:i=1, \dots, d \}|=|X_T\setminus S|$.  Invoking Lemma~\ref{lem:greedy_embedding_bipartite} gives an embedding of $T$.

\item [Case 2.2.2:]
Suppose that $T$ contains a matching of $k/100\Delta d\ge 6\Delta^5\cdot 2\varepsilon k \ge 6\Delta^5d$ leaves. 
Pick a submatching $N=\{y_1x_1, \dots, y_{3d}x_{3d}\}$ of $3d$ of these leaves that are at pairwise distance $\geq 4$ from each other and which are all in the same part of $T$. Assume that these edges are labelled so that $x_1, \dots, x_{3d}$ are the degree $1$ vertices. Let $X'_T\in \{X_T, Y_T\}$ be the part containing $x_1, \dots, x_{3d}$, and $Y_T'$ the other part, noting $k_2\ge |X'_T|$. Define $f:N\to G$ to map $N$ to $M$ and set $S=\{x_1, \dots, x_{3d}\}$.  These satisfy Lemma~\ref{lem:greedy_embedding_bipartite}: Indeed $|S|=3d\le 6\varepsilon k\le 32\varepsilon\Delta^2k$, $N_T(S)=\{y_1, \dots, y_{3d}\}\subseteq Y_T'$ by choice of $X_T'/Y_T'$, $f(N_T(S))\subseteq im_f\subseteq V(M)\subseteq Y'$ by construction, and $|X_K\setminus f(S)|=|X_K|=|X\cup X'|= \lceil k/2\rceil-d\ge  \frac{k_2+k_1-1}{2}-d = k_2-\frac{k_2-k_1-1}{2}-d-1 \ge k_2-2d-1\ge |X_T'|-3d= |X_T'\setminus S|$.
\end{itemize}
\end{itemize}
\end{itemize} 
\end{proof}

\section{Optimal bounds in Theorem~\ref{main_theorem}}
Here we prove Theorem~\ref{optimal_theorem} which sharpens Theorem~\ref{main_theorem}. The main ingredient for this is Theorem~\ref{main_theorem} itself, which is used as a block box. This is combined with some arguments for taking a graph with a $O(k)$-sized cover and outputting a graph with a $(1+o(1))k$-sized cover. The following lemma is one version of this. Its proof is a simpler version of Lemma~\ref{lem:stability_bipartite}.
\begin{lemma}\label{lem:thm_opt_sparse}
Let $1\gg \Delta^{-1}, q\gg \varepsilon\gg \gamma\gg  k^{-1}$.
Let $T$ be a tree with $\Delta(T)\leq \Delta$.
Let $G$ be a bipartite graph with parts $X,Y$, $|Y|\geq 400\varepsilon^{-1/2}k$ and  $\delta(G)\geq \varepsilon^{1/16} k/2$.  Let $D$ be a $q$-cut-dense  bipartite subgraph with parts $X_D, Y_D$ having $X\cap V(G)\subseteq X_D$, $|X_D|\le 10k$. If $G$ has no copy of $T$, then there is a subgraph $H$ of $G$ with $e(H)\ge e(G)- 2\varepsilon^{1/16}kn$ and  $|V(H)\cap X|\le k_1+10\gamma k$. 
\end{lemma}
\begin{proof}
Let $\Delta, \varepsilon\gg \gamma\gg   k^{-1}$. 
 If $|D|\geq 2q^{-1} k$, then by Fact~\ref{fact:cut_dense_min_degree}, $\delta(D)\geq (0.9 q)(2q^{-1} k)\ge 1.8k$ and Fact~\ref{fact:greedy_tree_embedding} gives a copy of $T$.  So assume  $|D|< 2q^{-1} k$. 
 Say a subgraph $H$ of $G'$ is good if $|V(H)\cap Y|= 400\varepsilon^{-1/8} k$, $|V(H)\cap X|\geq \varepsilon^{1/8} k/2$, and  $\delta_H(X\cap H)\geq 2k$. Let $H_1, \dots, H_s$ be  a maximal collection of good subgraphs which are vertex-disjoint in $X$. Let  $H=\bigcup H_i$.  
 Let $J$ be the subgraph on $V(G)$ consisting of edges of $G'$ disjoint from $V(H)\cap X$. 

We claim that $e(J)\leq \varepsilon^{1/8} k|Y|$. Indeed otherwise apply Lemma~\ref{lem:find_good_graph} to $J$ with $X=X, Y=Y, k'= \varepsilon^{1/8} k, C'=20\varepsilon^{-1/8}$ in order to get a subgraph with $|V(H)\cap Y|= (20\varepsilon^{-1/8})^2\varepsilon^{1/8}k=400\varepsilon^{-1/8}k$, $|V(H)\cap X|\geq \varepsilon^{1/8} k- |X|/(20\varepsilon^{-1/8})\overset{{|X|\leq 10k}}{\geq} \varepsilon^{1/8} k/2$, and  $\delta_H(X\cap H)\geq (20\varepsilon^{-1/8})\varepsilon^{1/8} k/4\geq 2k$. This is a good subgraph, contradicting maximality.
 
Note that since $|X|\leq 10k$  and $H_1, \dots, H_s$ are vertex disjoint in $X$, each having $\varepsilon^{1/8} k/2$ vertices in $X$, we have $s\le |X|/(\varepsilon^{1/8} k/2)\le 10k/(\varepsilon^{1/8} k/2)= 20\varepsilon^{-1/8}$ and so $H:=H_1\cup \dots \cup H_s$ has $|H\cup D|\leq |D|+(400\varepsilon^{-1/8} k)s \leq 2q^{-1}k+ (20\varepsilon^{-1/8})(400\varepsilon^{-1/8}k)\leq  (2q^{-1}+8010\varepsilon^{-1/4})k$. Also vertices in $H\setminus D\subseteq H\setminus X\subseteq Y$ have $\delta(G)\geq \varepsilon^{1/16}k$ neighbours in $D$. 
Note $|D|\geq |X|\geq \delta(G)\geq \varepsilon^{1/16}k$.
By  Lemma~\ref{lem:cut_dense_add_vertices}, we get that $H\cup D$ is $\frac{q(\varepsilon^{1/16}k/|D|)|D|^2}{4|H\cup D|^2}\geq \frac{q(\varepsilon^{1/16}k)^2}{4(2q^{-1}+8000\varepsilon^{-1/4})^2k^2}\geq \varepsilon^2 q^4$-cut-dense. 

 Now $|H\cap X|< k_1+\gamma k$, since otherwise   Lemma~\ref{lem:regularity_tree_embedding_bipartite} (applied to $H\cup D$ with $\varepsilon'=(\gamma/200)^2$) gives a copy of $T$. Letting $H'$ be the set of all edges touching $H\cap X$ satisfies the lemma because edges outside this are all in $J$  and $e(J)\le \varepsilon^{1/8} k|Y|$.
\end{proof}

Another part of the proof of Theorem~\ref{optimal_theorem} is to prove that theorem for dense graphs. This will use the following.
\begin{theorem}[R\"odl-Wysocka, \cite{rodl1997note}]\label{thm:Rodl_Wosocka}
Let $1/2\geq \gamma\gg n^{-1}$.
Every graph with $e(G)\geq \gamma n^2$ contains a $\lceil 0.4\gamma^3 n\rceil$-regular subgraph.
\end{theorem}
\begin{lemma}[\cite{hyperstability}]\label{lem:main_tree_embedding_lemma}
Let  $p, q,\varepsilon,\Delta^{-1}\gg  d^{-1}$
 and let $T$ be a order $d$ tree with $\Delta(T)\leq \Delta$. Let $G$ be a  $n$-vertex, $q$-cut-dense graph,  that contains an  $pn$-regular, order $\geq (2+\varepsilon)d$ subgraph $R$.  Then $G$ has a copy of $T$. 
\end{lemma}

The above lemma can be rephrased as the following lemma about matchings in the reduced graph.
\begin{lemma}\label{lem:2k_connected_matching_bound}
Let $M_0^{-1}, \Delta, \varepsilon, q\gg n^{-1}, k^{-1}$.
Let $T$ be a tree with $\Delta(T)\leq \Delta$ and $k$ vertices.
Let $G$ be a $q$-cut-dense graph and $R$ a $(\varepsilon, 5\sqrt{\varepsilon})$-reduced graph of $G$ with $|R|\leq M_0$. Suppose that $R$ has a matching $M$ with $e(M)\geq  (1+10\sqrt{\varepsilon})k|R|/n$ 
Then $G$ has a copy $T$.
\end{lemma}
\begin{proof}
Let $M_0^{-1}, \Delta, \varepsilon, q\gg p\gg n^{-1}, k^{-1}$. 
Let $M'$ be the bipartite subgraph of $G$ consisting of only the pairs in $M$, and let $H$ be the largest $pn$-regular subgraph of $M'$.

Consider some pair $(A,B)\in M$. We have $e(G[A,B]\setminus H)\le  3(pn/|G[A,B]\setminus H|)^{1/3}|G[A,B]\setminus H|^2$, as otherwise Theorem~\ref{thm:Rodl_Wosocka} would give a $\lceil0.4 (3(pn/|G[A,B]\setminus H|)^{1/3})^3 |G[A,B]\setminus H|\rceil \ge pn$-regular subgraph disjoint from $H$ (contradicting maximality of $H$). Also note that $|H\cap A|=|H\cap B|$ because $H[A,B]$ is regular and bipartite with parts $A,B$. These show that $|H\cap A|, |H\cap B|\le \varepsilon n/|R|$, as otherwise the fact that $(A,B)$ is $\varepsilon$-regular of density $\ge 5\sqrt{\varepsilon}$  would give $e(|H\cap A|, |H\cap B|)\ge (1- \varepsilon) 5\sqrt{\varepsilon}|H\cap A| |H\cap B|\ge (1- \varepsilon) 5\sqrt{\varepsilon} (\varepsilon n/|R|)^2>3(pn/(2n/|R|))^{1/3}(2n/|R|)^2>3(pn/|G[A,B]\setminus H|)^{1/3}|G[A,B]\setminus H|^2$. 

Using the bound from the previous paragraph, we have that $|H|=\sum_{(A,B)\in M}(|H\cap A|+|H\cap B|)\ge (1-\varepsilon)(2n/|R|)e(M)\ge  (1-\varepsilon)(2n/|R|)(1+10\sqrt{\varepsilon})k|R|/n\ge (1+9\sqrt{\varepsilon})2k$. Thus Lemma~\ref{lem:main_tree_embedding_lemma} applies to give a copy of $T$.
\end{proof}

Via K\"onig's Theorem we get the following.
\begin{lemma}\label{lem:bipartite_regularity_cover}
Let $\Delta^{-1}, q, \gg \varepsilon\gg  k^{-1}$.
Let $T$ be a tree with $\Delta(T)\leq \Delta$ and $k$ vertices.
Let $G$ be a  bipartite graph which is $q$-cut-dense with parts $X,Y$ with no copy of $T$. Then $G$ has a subgraph $G'$ with $e(G')\ge e(G)-10\sqrt{\varepsilon}n^2$ that has a cover of order $\le (1+10\sqrt{\varepsilon})k$. 
\end{lemma}
\begin{proof}
Let $\Delta^{-1}, q, \gg \varepsilon\gg m^{-1}\gg M^{-1}\gg  k^{-1}$.
Let $n_0:=|G|$. If $k\leq qn_0/4$, then note that Fact~\ref{fact:cut_dense_min_degree} gives $\delta(G)\geq 0.9qn$, and so we can greedily find $T$ using Fact~\ref{fact:greedy_tree_embedding}. Thus suppose $k> qn_0/4$.  
Use Szemer\'edi's Regularity Lemma to get a subgraph $G'$ of $G$ and a $(\varepsilon, 5\sqrt{\varepsilon})$-regularity partition with $s+t\in [m,M]$ parts  $X_1, \dots, X_s, Y_1, \dots, Y_t$  refining the partition $\{X, Y\}$, labelled so that $X_1, \dots, X_s\subseteq X, Y_1, \dots, Y_t\subseteq Y$. Let $n:=|G'|$, noting that $n\geq(1-\varepsilon)n_0$, which in particular gives $\Delta^{-1}, q, \varepsilon, m^{-1}, M^{-1}\gg n^{-1}$.
Let $R$ be the reduced graph of this partition. By Fact~\ref{fact:cut_dense_preserved_by_regularity}, $G'$ is $(q-2(\varepsilon+5\sqrt{\varepsilon}))\geq q/2$-cut-dense. By Lemma~\ref{lem:2k_connected_matching_bound},  $R$ has no matching of size $e(M)\ge (1+10\sqrt{\varepsilon})k|R|/n$. By K\"onig's Theorem, $R$ has a cover $C_R$ of order $\le (1+10\sqrt{\varepsilon})k|R|/n$. Let $C$ be the union of the parts giving $C_R$ to get a set of vertices with $|C|=|C_R|n/|R|\le  (1+10\sqrt{\varepsilon})(k|R|/n)(n/|R|)=(1+10\sqrt{\varepsilon})k$. Now $C$ covers all the edges of $G'$ and hence all but $(2\varepsilon+\sqrt{\varepsilon}) n^2\le 10\sqrt{\varepsilon}n^2$ of the edges of $G$ ($e(G)-e(G')\le (2\varepsilon+\sqrt{\varepsilon}) n^2$ by the remark after the statement of the Regularity Lemma). 
\end{proof}

Combining with Lemma~\ref{lem:bipartite_regularity_cover}, we can prove Theorem~\ref{optimal_theorem} for dense graphs.
\begin{lemma}\label{lem:thm_opt_dense}
Let $\Delta\gg \alpha\gg  k^{-1}$.
Let $T$ be a $k$-vertex tree with $\Delta(T)\leq \Delta$, and $G$ a $n$-vertex graph with no copies of $T$.
 Then  $G$ has a subgraph $H$ of $G$ with $e(H)\geq e(G)-3\alpha n^2$ whose connected components all have a cover of order $\le (1+\alpha)k$.
\end{lemma} 
\begin{proof}
Let $\Delta\gg \alpha\gg\varepsilon\gg  k^{-1}$.
Apply Lemma~\ref{lem:cut_dense_decomposition} in order to delete $\alpha n^2$ edges and get a  subgraph $G'$ whose connected components $G_1, \dots, G_t$ are $\alpha$-cut-dense. For each $G_i$ apply Lemma~\ref{lem:regularity_tree_embedding_nonbipartite} (with $q,\delta=\alpha$) in order to find a subgraph $G_i'$ of $G_i$ with $e(G_i')\ge e(G_i)-\alpha |G_i|^2$ satisfying (i) or (ii). If $G_i'$ satisfies (ii), let $G_i''=G_i'$, and if it satisfies (i) then apply Lemma~\ref{lem:bipartite_regularity_cover} to $G_i'$ to get $G_i''$. We claim that in either case $G_i''$ has a cover of order  $(1+\alpha)k$. Indeed if $G_i'$ satisfied (ii), then $C_1\cup C_2$ is such a cover since $|C_1\cup C_2|\le 2|C_1|+|C_2|\le (1+100\sqrt{\varepsilon})k\le (1+\alpha)k$. If $G_i'$ satisfied (i), then  Lemma~\ref{lem:bipartite_regularity_cover} gave us a cover of order $\le (1+10\sqrt{\varepsilon})k\le (1+\alpha)k$.
Now $\bigcup G_i''$ satisfies the lemma since the total number of edges missing from it is $\le \alpha n^2+\sum (\alpha +10\sqrt{\varepsilon})|G_i|^2\le \alpha n^2+2\alpha \sum|G_i|^2\le \alpha n^2+2\alpha (\sum|G_i|)^2\le  \alpha n^2+2\alpha n^2=3\alpha n^2$. 
\end{proof}

Combining lemmas~\ref{lem:thm_opt_sparse} and~\ref{lem:thm_opt_dense}, we get the following.
\begin{lemma}\label{lem:thm_opt_combine}
Let $\Delta\gg \beta\gg  k^{-1}$.
Let $T$ be a $k$-vertex tree with $\Delta(T)\leq \Delta$.
Let $G$ be a graph with a cover $D$ of order $|D|\leq 10k$. If $G$ has no copy of $T$, then there is a subgraph $H$ of $G$ with $e(H)\geq e(G)-3\beta k n$ whose connected components all have a cover of order $\le (1+\beta)k$.
\end{lemma}
\begin{proof}
Let $\Delta\gg \beta\gg p\gg q\gg \varepsilon\gg\gamma\gg n^{-1}, k^{-1}$.
Suppose $|G|\le \varepsilon^{-1} k$. Apply Lemma~\ref{lem:thm_opt_dense}, to get a subgraph $H$ with $e(H)\geq e(G)-3\gamma n^2$ whose connected components all have a cover of order $\le (1+\gamma)k$. Note that $3\gamma n^2\le 3\gamma \varepsilon^{-1} kn\le  3\beta kn$, and so this $H$ satisfies the lemma.

Suppose $|G|> \varepsilon^{-1} k$. Delete all $\leq 100k^2\leq 100\varepsilon kn\le \beta k n$ edges inside $D$ to get a bipartite graph $G[D,Y]$. Apply Lemma~\ref{lem:cut_dense_decomposition_of_dominated_graphs2} to delete $200p kn\leq \beta kn$ edges  in order to partition the remaining edges into  subgraphs $G_1, \dots, G_t$ such that each $G_i$ has a $p^{10^9p^{-8}}\ge q$-cut-dense subgraph $D_i$ with $|D_i|\leq  430p^{-8}k$ containing $X\cap G_i$ and 
 $|N(v)\cap V(D_i)|\geq p^{10^6p^{-1}}k\ge \varepsilon^{1/16}k$ for all $v\in G_i$.
 
For each $i$: if $|G_i\cap Y|\ge 400\varepsilon^{-1/2}k$ then apply Lemma~\ref{lem:thm_opt_sparse} to $G_i, D_i$ in order to get a subgraph $H_i$ with $e(H_i)\geq e(G_i)-4\varepsilon^{1/16}k|G_i|\ge e(G_i)-\beta k|G_i|$ and having a cover of order $\le (1+\gamma)k\le (1+\beta )k$. 
If $|G_i\cap Y|< 400\varepsilon^{-1/2}k$ then apply Lemma~\ref{lem:thm_opt_dense}  in order to get a subgraph $H_i$ with $e(H_i)\geq e(G_i)-\gamma |G_i|^2\ge e(G_i)-\gamma (400\varepsilon^{-1/2})k|G_i|\ge e(G_i)-\beta k|G_i|$ which has a cover of order $\le (1+\varepsilon)k\le  (1+\beta)k$.

Now $\bigcup H_i$ satisfies the lemma since the number of edges missing from this is $\le \beta kn +\beta kn+\beta k\sum |G_i|\le  3\beta kn$. 
\end{proof}
We now prove the main result of this section, which we restate below.
\begin{theorem*} Let $\Delta, \varepsilon\gg d^{-1}$. Let $T$ be a tree with $k$ edges and $\Delta(T) \leq \Delta$. For any graph $G$ with $e(G) \geq \varepsilon k |G|$ having no copies of $T$, it is possible to delete $\varepsilon e(G)$ edges to obtain a graph $H$, each of whose connected components has a cover of order $\leq (1 + \varepsilon)k$. \end{theorem*}
\begin{proof}
Pick $\Delta, \varepsilon\gg \gamma \gg d^{-1}$
Apply Theorem~\ref{main_theorem} in order to get a subgraph $G'$ with $e(G')\geq (1-\gamma)e(G)$ whose connected components $G_1, \dots, G_t$ have covers $D_1 \dots, D_t$ of order $10t$. To each $G_i$ apply Lemma~\ref{lem:thm_opt_combine} in order to get a subgraph $H_i$ having $e(H_i)\geq e(G_i)-3\gamma k |G_i|$, whose connected components have covers of order $\le (1+\gamma)k\le (1+\varepsilon)e(G)$. Now $\bigcup H_i$ satisfies the theorem because the number of missing edges is $\le \gamma e(G)+3\gamma k\sum |G_i|\le\gamma e(G)+ 3\gamma k|G|\le \gamma e(G)+ 3\gamma \varepsilon^{-1}e(G)=(\gamma+3\gamma \varepsilon^{-1})e(G)\le \varepsilon e(G)$.
\end{proof}

\subsubsection*{Acknowledgement}
The author would like to thank Kyriakos Katsamakis, Shoham Letzter, Benny Sudakov, Leo Versteegen, and Ella Williams for discussion about this project.
 
\bibliographystyle{abbrv} 
\bibliography{trees}  

\end{document}